\documentclass{article}

\title{Exact Recovery in the Geometric SBM}
\author{Julia Gaudio and Andrew Jin}

\usepackage[margin=1in]{geometry}
\usepackage{graphicx} 
\usepackage{amsfonts}
\usepackage{enumitem}
\usepackage{amsmath}
\usepackage{amssymb}
\usepackage{amsthm}
\usepackage{mathtools}
\usepackage{comment}
\usepackage{hyperref}
\usepackage[dvipsnames]{xcolor}
\usepackage{algorithm}
\usepackage[noend]{algpseudocode}
\usepackage{bbm}
\usepackage{soul}

\DeclareMathOperator*{\argmin}{argmin}

\newcommand{\R}{\mathbb{R}}
\renewcommand{\P}{\mathbb{P}}
\newcommand{\E}{\mathbb{E}}
\newcommand{\N}{\mathbb{N}}

\newcommand{\1}{\mathbbm{1}}
\newcommand{\calS}{\mathcal{S}}
\newcommand{\calN}{\mathcal{N}}
\newcommand{\calE}{\mathcal{E}}
\newcommand{\calA}{\mathcal{A}}
\newcommand{\calI}{\mathcal{I}}
\newcommand{\calJ}{\mathcal{J}}
\newcommand{\calH}{\mathcal{H}}

\newcommand{\calX}{\mathcal{X}}

\newcommand{\sign}{\text{sign}}
\newcommand{\visible}{\leftrightarrow{}}

\newcommand{\dist}{\nparallel_\epsilon}
\newcommand{\floor}[1]{\lfloor #1 \rfloor}

\theoremstyle{definition}
\newtheorem{definition}{Definition}
\newtheorem{assumption}{Assumption}

\theoremstyle{plain}
\newtheorem{theorem}{Theorem}
\newtheorem{lemma}{Lemma}
\newtheorem{prop}{Proposition}

\begin{document}

\maketitle

\begin{abstract}
Community detection is the problem of identifying dense communities in networks. Motivated by transitive behavior in social networks (``thy friend is my friend''), an emerging line of work considers spatially-embedded networks, which inherently produce graphs containing many triangles. In this paper, we consider the problem of exact label recovery in the Geometric Stochastic Block Model (GSBM), a model proposed by Baccelli and Sankararaman as the spatially-embedded analogue of the well-studied Stochastic Block Model. Under mild technical assumptions, we completely characterize the information-theoretic threshold for exact recovery, generalizing the earlier work of Gaudio, Niu, and Wei. 
\end{abstract}

\section{Introduction}
\label{sec:intro}
Community detection is a central problem in network analysis, with applications to social, biological, and physical networks. The problem has been well-studied in the statistics, probability, and machine learning literature through the lens of the \emph{Stochastic Block Model} (SBM), a canonical model for generating random graphs with community structure, proposed by Holland, Laskey, and Leinhardt \cite{Holland1983}. In the simplest version of the SBM on $n$ vertices, each vertex is assigned to one of two communities, Community $+$ and Community $-$, with equal probability and independently. Each pair of vertices in Community $+$ is connected by an edge with probability $p$, likewise for each pair of vertices in Community $-$. On the other hand, any pair of vertices in opposite communities is connected with probability $q$. Typically, one considers $p > q$ so that the graph exhibits assortative community structure. Variations of this model allow for multiple communities, nonuniform community assignment priors, and general edge probability parameters; see \cite{Abbe2018} for a survey.

Recently, a new direction in the study of community detection has emerged, which considers spatially-embedded graphs. Several models have been proposed, with the Geometric Stochastic Block Model (GSBM) \cite{Sankararaman2018} being the most similar to the standard SBM. The GSBM was developed in order to create networks which exhibit transitive behavior (``thy friend is my friend''), a typical feature of social networks \cite{Rapoport1953}.
In the GSBM, vertices are embedded into a $d$-dimensional cube of volume $n$ according to a Poisson process with intensity $\lambda$. Similarly to the standard SBM, the vertices are assigned to either Community $+$ or Community $-$ with equal probability and independently. Two functions, $f_{\text{in}}(\cdot)$ and $f_{\text{out}}(\cdot)$ govern the random formation of edges: a pair of vertices $u,v$ in at distance $\Vert u - v \Vert$ from each other is connected with probability $f_{\text{in}}(\Vert u - v \Vert)$ if they are in the same community, and otherwise is connected with probability $f_{\text{out}}(\Vert u - v \Vert)$. The geometric dependence leads to transitive community structure, favoring the formation of triangles, a behavior that is not seen in the standard SBM. 

In more detail, consider the SBM in the so-called \emph{logarithmic degree regime} with $p_n = \frac{a \log n}{n}$ and $q_n = \frac{b \log n}{n}$. An analogous GSBM model is derived by setting
 $f_{\text{in}}(x) = \frac{a}{\lambda \nu_d} \mathbbm{1}\{x \leq r_n\}$ and $f_{\text{out}}(x) = \frac{b}{\lambda \nu_d} \mathbbm{1}\{x \leq r_n\}$, where $r_n = (\log n)^{1/d}$ and $\nu_d$ is the volume of a $d$-dimensional unit ball. That is, a pair of vertices can only be connected if they lie within distance $r_n$ of each other; in that case, they are connected with probability either $\frac{a}{\lambda \nu_d}$ or $\frac{b}{\lambda \nu_d}$ depending on whether they are in the same or opposite communities. Setting the parameters in this way ensures that in both models, each vertex is connected to $\frac{a \log n}{2}$ vertices in the same community and $\frac{b \log n}{2}$ vertices in the opposite community, in expectation. However, the conditional edge probabilities are very different. This is seen by examining the conditional probability that two vertices $u,v$ form an edge, given that $u$ and $v$ are both connected to another vertex $w$. In the SBM, the probability that $u,v$ are connected by an edge given that $u,w$ and $v,w$ are both connected by edges is $\Theta(\log n/n)$. In the GSBM, the same conditional probability is $\Theta(1)$; if $u,w$ and $v,w$ form edges, then $u$ and $v$ are within distance $r_n$ of $w$, which means there is a constant probability that $u,v$ are themselves within distance $r_n$ of each other. Thus the two models exhibit dramatically different transitive behavior.

 In this paper, we revisit the GSBM as proposed by Baccelli and Sankararaman, treating the case of general edge probability functions $f_{\text{in}}(\cdot)$ and $f_{\text{out}}(\cdot)$. Under mild technical assumptions, we determine the information-theoretic (IT) threshold for exact recovery of the community labels. The IT threshold is governed by a generalized Chernoff--Hellinger divergence, quantifying the difference between $f_{\text{in}}(\cdot)$ and $f_{\text{out}}(\cdot)$. Moreover, we propose an efficient algorithm for exact recovery, extending the approach of Gaudio, Niu, and Wei \cite{Gaudio+2024}.

\subsection{Related Work}
The appearance of a Chernoff--Hellinger divergence is familiar in the study of community detection, following the prominent work of Abbe and Sandon \cite{AbbeSandon2015}. The present work introduces a generalized definition that accounts for distance dependence.

Several papers have followed on to the introduction of the GSBM, first by Abbe, Baccelli, and Sankararaman \cite{Abbe+2020}, who established an impossibility result for the GSBM in the logarithmic-degree regime, and proposed an efficient algorithm that was not shown to match the impossibility result. Gaudio, Niu, and Wei \cite{Gaudio+2024} subsequently proposed an efficient algorithm that matched the impossibility result of \cite{Abbe+2020}, thus establishing the IT threshold for the exact recovery problem in the GSBM. An extension of Gaudio, Guan, Niu, and Wei \cite{Gaudio+2025} treated the case of multiple communities, and general observation distribution (e.g. Gaussian pairwise observations), and established a corresponding IT threshold. Alongside these works, other spatial community detection models have been studied, including the Gaussian mixture block model \cite{Li2023}, the Geometric Block Model \cite{Galhotra2018}, and the Soft Geometric Block Model \cite{Avrachenkov2021}.

This paper addresses a gap in the study of the GSBM, for which the focus has been on the case where $f_{\text{in}}(\cdot)$ and $f_{\text{out}}(\cdot)$ are step functions of the form $f_{\text{in}}(x) = a \mathbbm{1}\{x \leq r_n\}$, $f_{\text{out}}(x) = b \mathbbm{1}\{x \leq r_n\}$. The general version was studied by Avrachenkov, Kumar, and Leskel\"a \cite{Avrachenkov2024} for $d=1$, under the assumption that $f_{\text{in}}(\cdot)$ and $f_{\text{out}}(\cdot)$ are multiples of each other. In this paper, we do not require this assumption, and consider any constant dimension $d \geq 1$.

\subsection{Notation and Organization}
\label{subsec:notation}
We write $[n] := \{1, 2, \ldots, n\}$. We use $D_+(\cdot \Vert \cdot)$ to denote the Chernoff-Hellinger (CH) divergence, which was first introduced in \cite{AbbeSandon2015} and extended in \cite{Gaudio+2025}. To define the CH-divergence, let $p = (p_1, \ldots, p_k)$ and $q = (q_1, \ldots, q_k)$ be vectors of probability distributions, where $p_i(\cdot)$ and $q_i(\cdot)$ denote probability mass functions (in the discrete case) or probability density functions (in the continuous case). In addition, let $\pi = (\pi_1, \ldots, \pi_k)$ be a vector of prior probabilities. Then the CH-divergence between $p$ and $q$ is defined
\begin{equation}
    D_+(p \Vert q; \pi) = 1 - \inf_{t \in [0,1]} \sum_{i = 1}^k \pi_i \sum_{x \in \mathcal{X}} p_i(x)^t q_i(x)^{1-t} \label{eq:CH-formula}
\end{equation}
when $p$ and $q$ are discrete, and 
\begin{equation*}
    D_+(p \Vert q; \pi) = 1 - \inf_{t \in [0,1]} \sum_{i = 1}^k \pi_i \int_{x \in \mathcal{X}} p_i(x)^t q_i(x)^{1-t} \: dx
\end{equation*}
when $p$ and $q$ are continuous. We extend this definition to the case where $p_i$ and $q_i$ are functions of an additional parameter $y \in \mathcal{Y}$, which we therefore denote as $p_i(\cdot; y)$ and $q_i(\cdot; y)$. Let $g(y)$ be a density on $\mathcal{Y}$. Then, we define the CH-divergence between $p$ and $q$ as
\begin{equation}
\label{eq:CH_divergence_discrete}
    D_+(p \Vert q; \pi, g) = 1 - \inf_{t \in [0,1]} \sum_{i = 1}^k \pi_i \int_{y \in \mathcal{Y}} g(y)\sum_{x \in \mathcal{X}} p_i(x;y)^t q_i(x;y)^{1-t} \: dy
\end{equation}
when $p$ and $q$ are discrete, and 
\begin{equation}
\label{eq:CH_divergence_continuous}
    D_+(p \Vert q; \pi, g) = 1 - \inf_{t \in [0,1]} \sum_{i = 1}^k \pi_i \int_{y \in \mathcal{Y}} g(y) \int_{x \in \mathcal{X}} p_i(x;y)^t q_i(x;y)^{1-t} \: dx \: dy
\end{equation}
when $p$ and $q$ are continuous. 

The rest of the paper is organized as follows. Section \ref{sec:model} formally defines the Geometric Stochastic Block Model and states our main result establishing an information-theoretic threshold for exact recovery under suitable assumptions. Section \ref{sec:exact_recovery_algorithm} describes our algorithm to achieve exact recovery above the threshold and presents a proof outline. Sections \ref{sec:block_size} to \ref{sec:refine} prove that our algorithm achieves exact recovery above the threshold. Section \ref{sec:Impossibility} proves that exact recovery is impossible below the threshold. Section \ref{sec:conclusion} discusses future research directions.

\section{Model and Main Results}
\label{sec:model}
We now formally define the GSBM, as introduced by Baccelli and Sankararaman \cite{Sankararaman2018}.
\begin{definition}
    Let $\lambda > 0$, $r > 0$ be constants, and let $d \in \N$. Let $f_\text{in}, f_\text{out}: \R \to [0, 1]$ be functions satisfying the regularity conditions in Assumption \ref{assump:regularity_conditions}. Then, we sample a graph $G$ from $\text{GSBM}(\lambda, r, n, f_\text{in}, f_\text{out}, d)$ using the following procedure.
    \begin{enumerate}
        \item Vertices are distributed in the region $\calS_{d,n} := [-n^{1/d}/2, n^{1/d}/2]^d \subset \R^d$ via a Poisson point process with intensity parameter $\lambda$. That is, we first sample the number of vertices from a $\text{Pois}(\lambda n)$ distribution, then distribute the vertices in $\calS_{d,n}$ uniformly at random. Let $V$ denote the set of all vertices.

        \item Each vertex is independently assigned a community label from $\{+1, -1\}$ uniformly at random. We denote the community label of vertex $u$ as $\sigma^*(u)$.

        \item Conditioned on the locations and community labels, edges are formed independently between each pair of vertices. Let $E$ denote the set of all edges, and define the functions $\overline{f}_\text{in}(x) := f_\text{in}(x / (\log n)^{1/d}), \overline{f}_\text{out} := f_\text{out}(x / (\log n)^{1/d} )$. For each pair of vertices $\{u, v\}$ where $u \neq v$, an (undirected) edge is formed between the vertices with probability
        \begin{equation*}
            \P(\{u, v\} \in E) = 
            \begin{dcases}
                \overline{f}_\text{in}(\|u-v\|) \text{ if } \sigma^*(u) = \sigma^*(v) \\
                \overline{f}_\text{out}(\|u-v\|) \text{ if } \sigma^*(u) \neq \sigma^*(v)
            \end{dcases}
        \end{equation*}
        where $\| \cdot \|$ is the Euclidean toroidal metric defined as
        \begin{equation}
        \label{eq:def_toroidal_metric}
            \|u - v\| := \sqrt{\sum_{i = 1}^d \Big( \min\big\{|u_i - v_i|, n^{1/d} - |u_i - v_i|\big\} \Big)^2}.
        \end{equation}
        If $\{u, v\} \in E$, we write $u \sim v$. Otherwise, we write $u \nsim v$.
    \end{enumerate}
\end{definition}

We assume the following regularity conditions on $f_\text{in}$ and $f_\text{out}$. 
\begin{assumption}
\label{assump:regularity_conditions}
    Let $r := \inf\{ x > 0: f_\text{in}(t), f_\text{out}(t) = 0  \: \forall t > x \}$. 
    
    \begin{enumerate}[label=(\roman*)]
        \item \label{eq:function_assumption_1} We assume that $r < \infty$.

        \item \label{eq:function_assumption_2} There exists a constant $\xi > 0$ such that
        \begin{equation*}
            \xi < f_\text{in}(t), f_\text{out}(t) < 1 - \xi \: \text{ for } \:  t \leq r.
        \end{equation*}
        
        \item \label{eq:function_assumption_3} The set $\{0 \leq t \leq r: f_\text{in}(t) = f_\text{out}(t)\}$ is finite. That is,
        \begin{equation*}
            \{0 \leq t \leq r \mid f_\text{in}(t) = f_\text{out}(t)\} = \{t_1, t_2, \ldots, t_m\} \: \text{ for } \: t_1 < t_2 < \ldots < t_m.
        \end{equation*}

        \item \label{eq:function_assumption_4} Define the function
        \begin{equation}
        \label{eq:def_gamma}
            \gamma(\epsilon) = \sup \Big\{ \text{dist}(t, \{t_1, t_2, \ldots, t_m\}) : |f_\text{in}(t) - f_\text{out}(t) | \leq \epsilon, 0 \leq t \leq r \Big\}
        \end{equation}
        where 
        \begin{equation*}
            \text{dist}(t, \{t_1, t_2, \ldots, t_m\}) = \min_{i \in [m]}\{|t - t_i| \}.
        \end{equation*}
         That is, $\gamma(\epsilon)$ measures the largest distance between a point $t \in [0,r]$ and the set of intersection points $\{t_1, \dots, t_m\}$, among those points $t$ for which $f_\text{in}(t)$ and $f_\text{out}(t)$ are within $\epsilon$ of each other. Then, we assume that $\gamma(\epsilon) \to 0$ as $\epsilon \to 0$. We remark that this assumption is satisfied when $f_\text{in}$ and $f_\text{out}$ are continuous.
    \end{enumerate}
\end{assumption}

Assumption \ref{eq:function_assumption_1} states that we can only form edges between two vertices $u$ and $v$ which are ``close together'', meaning that $\|u - v\| \leq r (\log n)^{1/d}$. In this case, we say that $u$ \textit{is visible to} $v$, which we denote as $u \visible v$. In addition, we will sometimes be interested in considering all points visible to a vertex $v$. Hence, we define the \textit{neighborhood} of $v$ as $\calN(v) := \{u \in \R^d : \|u-v\| < r(\log n)^{1/d} \}$. Finally, we call this threshold  $r (\log n)^{1/d}$ the \textit{visibility radius}.

To interpret Assumption \ref{eq:function_assumption_4}, note that \eqref{eq:def_gamma} implies
\begin{equation*}
\Big\{ 0 < t < r: |f_\text{in}(t) - f_\text{out}(t)| \leq \epsilon \Big\} \subseteq \bigcup_{i=1}^m \Big[ t_i-\frac{\gamma(\epsilon)}{2}, t_i + \frac{\gamma(\epsilon)}{2} \Big].
\end{equation*}
Therefore, Assumption \ref{eq:function_assumption_4} means that when $\epsilon$ is sufficiently small, then the regions where $|f_\text{in}(t) - f_\text{out}(t)| \leq \epsilon$ are confined within small intervals around $t_1, t_2, \ldots, t_m$. This ensures that $f_\text{in}$ and $f_\text{out}$ are significantly different (i.e. that $|f_\text{in}(t) - f_\text{out}(t)| > \epsilon$) in most of their domain. To motivate this assumption, observe that we must determine the label of vertex $v$ by looking at whether it has an edge to another vertex $u$. For this to give meaningful information, we need $\overline{f}_\text{in}(\|u - v\|)$ and $\overline{f}_\text{out}(\|u-v\|)$ to be significantly different. Thus, Assumption \ref{eq:function_assumption_4} ensures that most other vertices give meaningful information.

Our goal is to achieve exact recovery of the community labels, which we now define. Suppose that $\sigma^*_n$ is the true label of each vertex and $\tilde{\sigma}_n$ is an estimated label of each vertex. We define the agreement of $\tilde{\sigma}_n$ and $\sigma^*_n$ as 
\begin{equation*}
    A(\tilde{\sigma}_n, \sigma^*_n) = \frac{1}{|V|} \max \Big\{\sum_{u \in V} \1_{\tilde{\sigma}_n(u) = \sigma^*_n(u)}, \sum_{u \in V} \1_{\tilde{\sigma}_n(u) = - \sigma^*_n(u)} \Big\},
\end{equation*}
which is the proportion of vertices that $\tilde{\sigma}_n$ labels correctly, up to a global sign flip. Then, we say that $\tilde{\sigma}_n$ achieves
\begin{itemize}
    \item \textit{exact recovery} if $\lim_{n \to \infty} \P(A(\tilde{\sigma}_n, \sigma^*_n) = 1) = 1$,
    \item \textit{almost-exact recovery} if $\lim_{n \to \infty} \P(A(\tilde{\sigma}_n, \sigma^*_n) \geq 1 - \epsilon) = 1$ for all $\epsilon > 0$, and
    \item \textit{partial recovery} if $\lim_{n \to \infty} \P(A(\tilde{\sigma}_n, \sigma^*_n) \geq \alpha) = 1$ for some constant $\alpha > 0$.
\end{itemize}
In other words, exact recovery means that $\tilde{\sigma}_n$ labels every vertex correctly (up to a global sign flip), with probability going to 1 as the graph sizes goes to infinity. We omit the subscript $n$ in the remainder of the paper for clarity.

In this paper, we show that there exists an information-theoretic threshold for $G \sim \text{GSBM}(\lambda, r, n, f_\text{in}, f_\text{out}, d)$, above which there is an efficient algorithm for exact recovery and below which it is impossible to achieve exact recovery. Define the information metric
\begin{equation}
\label{eq:information_metric}
    I(f_\text{in}, f_\text{out}) := \lambda \nu_d r^d \int_0^r \left( 1- \sqrt{f_\text{in}(t) f_\text{out}(t)} - \sqrt{(1-f_\text{in}(t))(1-f_\text{out}(t))} \right) \frac{dt^{d-1}}{r^d} dt
\end{equation}
where $\nu_d$ is the volume of the $d$-dimensional unit sphere.

\begin{theorem}[Achievability]
\label{thm:achievability}
    There exists a polynomial-time aglorithm achieving exact recovery in $G \sim \text{GSBM}(\lambda, r, n, f_\text{in}, f_\text{out}, d)$ when
    \begin{enumerate}
        \item $d = 1$, $\lambda r > 1$, and $I(f_\text{in}, f_\text{out}) > 1$, or
        \item $d \geq 2$ and $I(f_\text{in}, f_\text{out}) > 1$.
    \end{enumerate}
\end{theorem}

\begin{theorem}[Impossibility]
\label{thm:impossibility}
    Any estimator fails to achieve exact recovery in $G \sim \text{GSBM}(\lambda, r, n, f_\text{in}, f_\text{out}, d)$ when
    \begin{enumerate}
        \item $d = 1$ and $\lambda r < 1$, or
        \item $I(f_\text{in}, f_\text{out}) < 1$.
    \end{enumerate}
\end{theorem}

From Theorems \ref{thm:achievability} and \ref{thm:impossibility}, we obtain the information-theoretic threshold for exact recovery in $G \sim \text{GSBM}(\lambda, r, n, f_\text{in}, f_\text{out}, d)$, when $f_\text{in}$ and $f_\text{out}$ satisfy Assumption \ref{assump:regularity_conditions}.

We remark that we can express the information metric (\ref{eq:information_metric}) in terms of a CH-divergence as defined in \eqref{eq:CH_divergence_discrete}. Let $p(y) = (p_1(\cdot;y), p_2(\cdot;y))$ and $q(y) = (q_1(\cdot;y), q_2(\cdot;y))$ be the probability mass functions of Bernoulli random variables which correspond to edge probabilities conditioned on community assignments and distance, where $p_1(\cdot; y), q_2(\cdot; y) \sim \text{Bern}(f_\text{in}(y))$ and $p_2(\cdot; y), q_1(\cdot; y) \sim \text{Bern}(f_\text{out}(y))$ for $y \in [0, r]$. Furthermore, let $\pi = (\pi_1, \pi_2)$ where $\pi_1 = \pi_2 = 1/2$ and $g(y) = dy^{d-1}/r^d$, which is a density on $[0, r]$. Then, we can write the information metric (\ref{eq:information_metric}) as 
\begin{equation*}
    I(f_\text{in}, f_\text{out}) = \lambda \nu _d r^d D_+(p \Vert q; \pi, g).
\end{equation*}

Intuitively, the only information we have to distinguish between two communities are their edge probabilities to each community, which we can think of as their ``behavior''. Thus, the CH-divergence between $p$ and $q$ measures how different the two communities are based on their edge probabilities to each community. For exact recovery to be possible, the two communities must have sufficiently different behavior for them to be distinguishable based on the edge observations, which gives rise to the CH-divergence term in the information-theoretic threshold. We remark that when $f_\text{in} = f_\text{out}$, then the behavior is identical between the two communities; this is reflected in the CH-divergence as $D_+(p \Vert q; \pi, g) = 0$ in this case. We also see that when $f_\text{in}$ and $f_\text{out}$ are close to each other, then the CH-divergence is small, which means that a larger value of $\lambda$ (i.e. more vertices) is required to achieve exact recovery.

We also note that our notion of CH divergence given in \eqref{eq:CH-formula} can be used to characterize the earlier result of Abbe and Sandon \cite{AbbeSandon2015} for the standard SBM; here we focus on the symmetric SBM with two communities, as it is analogous to the present setting. 
The symmetric SBM with two communities (in the logarithmic degree regime) is defined as taking $n$ vertices, assigning each vertex as community $+1$ with probability $\pi_1$ and community $-1$ with probability $\pi_{-1}$, then forming an edge between each pair of vertices with probability $\frac{a \log n}{n}$ if the two vertices are from the same community or with probability $\frac{b \log n}{n}$ if the two vertices are from different communities.  Abbe and Sandon \cite{AbbeSandon2015} showed that the IT threshold for exact recovery is
\[\sup_{t \in [0,1]} \pi_1\left(ta + (1-t)b - a^t b^{1-t}\right) +\pi_{-1}\left(tb + (1-t)a - b^t a^{1-t}\right)  = 1, \]
where the left hand side of the above expression is referred to as a CH divergence.
Alternatively, the IT threshold can be stated as $\frac{n}{\log n} D_+(p \Vert q; \pi) = 1$ where $p = \left( \frac{a \log n}{n}, 1 - \frac{a \log n}{n}\right)$ and $q = \left( \frac{b \log n}{n}, 1 - \frac{b \log n}{n}\right)$.\footnote{See also \cite{Dreveton2023} for another characterization of the IT threshold for the standard SBM in terms of R\'enyi divergences.} To see this, observe that the definition yields
\begin{align*}
\frac{n}{\log n} D_+(p \Vert q; \pi) &= \frac{n}{\log n} \left\{1 - \inf_{t \in [0,1]} \left[\pi_1 \left[\left(\frac{a \log n}{n}\right)^t \left(\frac{b \log n}{n}\right)^{1-t} + \left(1-\frac{a \log n}{n}\right)^t \left(1-\frac{b \log n}{n}\right)^{1-t} \right] \right. \right.\\
& \quad \quad \quad \quad \quad \quad \quad \quad    \left.\left. +\pi_{-1} \left[\left(\frac{b \log n}{n}\right)^t \left(\frac{a \log n}{n}\right)^{1-t} + \left(1-\frac{b \log n}{n}\right)^t \left(1-\frac{a \log n}{n}\right)^{1-t} \right] \right] \right\}.
\end{align*}
Since $\left(1 - \frac{x \log n}{n}\right)^y = 1 - \frac{xy \log n}{n} + o\left( \frac{\log n}{n}\right)$ for $x, y > 0$, we have
\begin{align*}
\frac{n}{\log n} D_+(p \Vert q; \pi) &= \frac{n}{\log n} \left\{1 - \inf_{t \in [0,1]} \left[\pi_1 \left[\left(\frac{a \log n}{n}\right)^t \left(\frac{b \log n}{n}\right)^{1-t} +1 -\frac{at \log n}{n} - \frac{(1-t)b \log n}{n} + o \left(\frac{\log n}{n}\right)\right] \right. \right.\\
& \quad \quad \quad \quad \quad \quad \quad \quad    \left.\left. +\pi_{-1} \left[\left(\frac{b \log n}{n}\right)^t \left(\frac{a \log n}{n}\right)^{1-t}  +1 -\frac{tb \log n}{n} - \frac{(1-t)a \log n}{n} + o \left(\frac{\log n}{n}\right) \right] \right] \right\}\\
&= \sup_{t \in [0,1]} \pi_1 \left[ta + (1-t)b -a^tb^{1-t}  \right] + \pi_{-1} \left[tb + (1-t)a -b^ta^{1-t}  \right] + o\left(1\right).
\end{align*}

\section{Exact Recovery Algorithm} 
\label{sec:exact_recovery_algorithm}
This section presents our algorithm to achieve exact recovery, Algorithm \ref{alg:exact_recovery}. We will use a two-phase approach for exact recovery, adapting the procedure developed in \cite{Gaudio+2024}. The first phase constructs an estimator $\hat{\sigma}$ that achieves almost-exact recovery. We start by partitioning $\calS_{d, n}$ into $\floor{n / (r^d \chi \log n)}$ cubes of volume $r^d \chi \log n$ for an appropriate constant $\chi > 0$, which we call \textit{blocks}. For convenience, we will omit the floor function in the rest of the paper when discussing the number of blocks. Then, we label the vertices block-by-block for all blocks with at least $\delta \log n$ vertices, where $\delta > 0$ is a suitable constant. We first label the vertices in an initial block, which is done using \texttt{Pairwise Classify} (Algorithm \ref{alg:pairwise_classify}). Then, we use the labeled vertices in one block to estimate the labels of the vertices in another block using the procedure described in \texttt{Propagate} (Algorithm \ref{alg:propagate}), until all blocks with at least $\delta \log n$ vertices are labeled. The second phase then uses the almost-exact estimator $\hat{\sigma}$ to produce an estimator $\tilde{\sigma}$ which achieves exact recovery, revising the label of each vertex using \texttt{Refine} (Algorithm \ref{alg:refine}).

To make this notion of labeling vertices block-by-block precise, we use the following definitions from \cite{Gaudio+2024}. Let $B_i$ be the $i$th block and $V_i$ be the set of vertices in $B_i$ for $i \in [n/(r^d \chi \log n)]$. 

\begin{definition}[Occupied Block]
    Given $\delta > 0$, we call a block $B_i$ $\delta$-occupied if $B_i$ contains at least $\delta \log n$ vertices. Otherwise, we call $B_i$ $\delta$-unoccupied.
\end{definition}

\begin{definition}[Visible Blocks]
    We call a pair of blocks $B_i$ and $B_j$ mutually visible, which we denote $B_i \visible B_j$, if
    \begin{equation*}
        \sup_{x \in B_i, y \in B_j} \|x - y\| \leq r(\log n)^{1/d}.
    \end{equation*}
\end{definition}

\begin{definition}[Block Visibility Graph]
    Suppose that we have a graph $G = (V,E)$ on $\calS_{d, n}$ and a partition of $\calS_{d, n}$ into blocks of volume $v(n)$. We define the $(v(n), c(n))$-visibility graph of $G$ as the graph $H = (V^\dagger, E^\dagger)$, where $V^\dagger := \{i \in [n / v(n)] : |V_i| \geq c(n)\}$ is the set of all blocks with at least $c(n)$ vertices and $E^\dagger := \{\{i, j\}: i, j \in V^\dagger, B_i \visible B_j\}$ consists of all pairs of blocks in $V^\dagger$ which are mutually visible. We will also call $H$ the block visibility graph of $G$.
\end{definition}

From these definitions, we see that Phase I attempts to label all blocks which are $\delta$-occupied. To accomplish this, we first construct the $(r^d \chi \log n, \delta \log n)$-visibility graph of $G$, which we call $H$. If $H$ is connected, then we can find a spanning tree of $H$, which gives us the following strategy for labeling all $\delta$-occupied blocks. First, we label the block at the root of the tree, which we call $B_{i_1}$, using \texttt{Pairwise Classify}. Then, we can estimate the labels of its child blocks using \texttt{Propagate}, and so on until all $\delta$-occupied blocks are labeled. The conditions $\lambda r > 1$ (when $d = 1$) and $\lambda \nu_d r^d > 1$ (when $d \geq 2$) ensure that $H$ is connected with high probability, which allows for this propagation scheme to work.

\begin{algorithm}
\caption{\texttt{Exact Recovery}}
\label{alg:exact_recovery}
\begin{algorithmic}[1]
\Require $G \sim \text{GSBM}(\lambda, r, n, f_\text{in}, f_\text{out}, d)$
\Ensure An estimated community labeling $\tilde{\sigma}: V \to \{+1, -1\}$
\State \textbf{Phase I:}
\State Take $\chi, \delta > 0$ which satisfy \eqref{eq:chi_delta_conditions_d=1} if $d = 1$ or \eqref{eq:chi_delta_conditions_d>=2} if $d \geq 2$. \label{line:set_chi_delta}
\State Partition $S_{d, n}$ into $n/(r^d \chi \log n)$ blocks of volume $r^d \chi \log n$. Let $B_i$ be the $i$th block and $V_i$ be the set of vertices in $B_i$ for $i \in [n/(r^d \chi \log n)]$. \label{line:create_blocks}
\State Construct the $(r^d \chi \log n, \delta \log n)$-visibility graph $H = (V^\dagger, E^\dagger)$ of $G$. \label{line:visibility_graph}
\If{$H$ is disconnected}
    \State Return \texttt{FAIL}
\EndIf
\State Find a rooted spanning tree of $H$, ordering $V^\dagger = \{i_1, i_2, \ldots, i_{|V^\dagger|}\}$ in breadth-first order.
\State Apply \texttt{Pairwise Classify} (Algorithm \ref{alg:pairwise_classify}) on input $(G, V_{i_1}, f_\text{in}, f_\text{out})$ to obtain a labeling $\hat{\sigma}$ of $V_{i_1}$. \label{line:pairwise_classify}
\State Choose $\epsilon > 0$ which satisfies Lemma \ref{lem:distinguishing_block_lemma}.
\For{$j = 2, \ldots, |V^\dagger|$} \label{line:interate_propagate}
    \State Apply \texttt{Propagate} (Algorithm \ref{alg:propagate}) on input $(G, V_{p(i_j)}, V_{i_j}, f_\text{in}, f_\text{out})$ to obtain a labeling $\hat{\sigma}$ of $V_{i_j}$. \label{line:propagate}
\EndFor
\For{$v \in V \setminus \cup_{i \in V^\dagger} V_i$}
    \State Set $\hat{\sigma}(v) = 0$.
\EndFor
\State \textbf{Phase II:}
\For{$v \in V$} \label{line:iterate_refine}
    \State Apply \texttt{Refine} (Algorithm \ref{alg:refine}) on input $(G, \hat{\sigma}, v)$ to compute $\tilde{\sigma}(v)$. \label{line:refine}
\EndFor
\end{algorithmic}
\end{algorithm}

\begin{algorithm}
\caption{\texttt{Pairwise Classify}} 
\label{alg:pairwise_classify}
\begin{algorithmic}[1]
\Require Graph $G = (V, E)$, vertex set $S \subset V$, functions $f_\text{in}, f_\text{out}$ satisfying Assumption \ref{assump:regularity_conditions}.
\State Choose an arbitrary vertex $u_0 \in S$, set $\hat{\sigma}(u_0) = +1$.
\For{$v \in S \setminus \{u_0\}$}
    \State Compute $X_v$ using (\ref{eq:def_Xv}).
    \State Set $\hat{\sigma}(v) = \sign(X_v)$.
\EndFor
\end{algorithmic}
\end{algorithm}

\begin{algorithm}
\caption{\texttt{Propagate}}
\label{alg:propagate}
\begin{algorithmic}[1]
\Require Graph $G = (V, E)$, disjoint sets of vertices $S, S' \subset V$ which are mutually visible, where $S$ is labeled according to $\hat{\sigma}$, fixed $\epsilon > 0$.
\If{$|\{u \in S: \hat{\sigma}(u) = +1, u \dist v \}| \geq |\{u \in S: \hat{\sigma}(u) = -1, u \dist v\}|$}
    \For{$v \in S'$}
        \State Compute $Y_v$ using (\ref{eq:def_Yv_more_positive}).
        \State Set $\hat{\sigma}(v) = \sign(Y_v)$.
    \EndFor
\Else
    \For{$v \in S'$}
        \State Compute $Y_v$ using (\ref{eq:def_Yv_more_negative}).
        \State Set $\hat{\sigma}(v) = \sign(Y_v)$.
    \EndFor
\EndIf
\end{algorithmic}
\end{algorithm}

\begin{algorithm}
\caption{\texttt{Refine}}
\label{alg:refine}
\begin{algorithmic}[1]
\Require Graph $G \sim \text{GSBM}(\lambda, r, n, f_\text{in}, f_\text{out}, d)$, vertex $v \in V$, labeling $\hat{\sigma}: V \to \{-1, 0 ,+1\}$.
\Ensure An estimated labeling $\tilde{\sigma}(v) \in \{-1, +1\}$.
\State Compute $\tau(v, \hat{\sigma})$ using (\ref{eq:def_tau}).
\State Set $\tilde{\sigma}(v) = \sign(\tau(v, \hat{\sigma}))$.
\end{algorithmic}
\end{algorithm}

\subsection{The Algorithm}
We now describe the main components of Algorithm \ref{alg:exact_recovery}. 

The first step is labeling the initial block $B_{i_1}$ using \texttt{Pairwise Classify}. Let $V_{i_1}$ be the vertices in $B_{i_1}$. Since the model is symmetric, we only need to find the correct partition of $V_{i_1}$ into two communities. Therefore, we can take an arbitrary vertex $u_0 \in V_{i_1}$ and set $\hat{\sigma}(u_0) = +1$.  Now, let $L_{i_1}$ denote the locations of the vertices $V_{i_1}$ and consider a vertex $v \neq u_0$ in $V_{i_1}$. To label $v$, we compute 
\begin{equation}
\label{eq:def_Xv}
    X_v := \sum_{u \in V_{i_1} \setminus\{u_0, v\}} \alpha_u \Big(\mathbbm{1}\{u \sim u_0, u \sim v\} - \mathbb{P}(u \sim u_0, u \sim v \mid  L_{i_1} )\Big),
\end{equation}
where $\alpha_u$ is a sign correction defined as
\begin{equation*}
    \alpha_u := \sign\Big[ \mathbb{P}\left(u \sim u_0, u \sim v \mid \sigma^*(v) = \sigma^*(u_0), L_{i_1} \right) - \mathbb{P}\left(u \sim u_0, u \sim v \mid \sigma^*(v) \neq \sigma^*(u_0), L_{i_1} \right)\Big].
\end{equation*}
Then, we set the estimated label for $v$ as 
\begin{equation}
\label{eq:sigma_hat_initial}
\hat{\sigma}(v) = \sign (X_v).
\end{equation}
Intuitively, $X_v$ measures how much evidence we have for the label $v$ being $+1$ or $-1$. First, $\alpha_u$ determines whether the existence of a common neighbor $u$ of both $u_0$ and $v$ gives evidence for $v$ being in the same community as $u_0$ (in which case $\alpha_u = 1$) or $v$ being in a different community than $u_0$ (in which case $\alpha_u = -1$), depending on the likelihoods of the two cases. Then, $X_v$ measures the amount of evidence we have by comparing the observed number of common neighbors with the expected number of common neighbors (conditioned on the vertex locations). We will later show that the conditional expectation $\E[X_v \mid \sigma^*(v)]$ is positive when $\sigma^*(v) = \sigma^*(u_0)$ and negative when $\sigma^*(v) \neq \sigma^*(u_0)$. Hence, the label of $v$ can be inferred from the sign of $X_v$.

Next, we discuss \texttt{Propagate}, which labels a block $B_i$ with vertices $V_i$ using the estimated labels of its parent block $B_{p(i)}$ with vertices $V_{p(i)}$. We first need the following definition.
\begin{definition}[Distinguishing Point]
    We say that a point $u$ $\epsilon$-distinguishes a point $v$, which we denote $u \dist v$, if 
    \begin{equation}
        \left| \overline{f}_\text{in}(\|u-v\|) - \overline{f}_\text{out}(\|u-v\|) \right| > \epsilon.
    \end{equation}
    Furthermore, we say that a point $u$ $\epsilon$-distinguishes a set of points $S$ if $u \dist v$ for all $v \in S$.
\end{definition}
To motivate this definition, suppose that we know the community label of $u$ and want to find the community label of $v$. Then, we can examine whether $v$ has an edge with $u$ and determine if that gives evidence for $v$ being in the same or opposite community compared to $u$. We can do this by comparing the likelihoods of the two cases, which are given by $\overline{f}_\text{in}(\|u-v\|)$ and $\overline{f}_\text{out}(\|u-v\|)$. However, this requires that $\overline{f}_\text{in}(\|u-v\|)$ and $\overline{f}_\text{out}(\|u-v\|)$ be sufficiently different---that is, $u$ $\epsilon$-distinguishes $v$ for some suitable value of $\epsilon$.

Now, we discuss how to label each vertex in $V_i$. Let $\epsilon > 0$. If 
\begin{equation*}
    \Big|\{u \in V_{p(i)}: \hat{\sigma}(u) = +1, u \dist v \}\Big| \geq \Big|\{u \in V_{p(i)}: \hat{\sigma}(u) = -1, u \dist v\}\Big|,
\end{equation*}
meaning that $\hat{\sigma}$ labels more vertices in $V_p(i)$ which $\epsilon$-distinguish $v$ as $+1$ than $-1$, then we compute 
\begin{equation}
\label{eq:def_Yv_more_positive}
    Y_v := \sum_{u \in V_{p(i)}: \hat{\sigma}(u) = +1, \text{$u \dist v$}} \beta_u \Big( \1(u \sim v) - \P(u \sim v \mid L_{p(i)}, L_i) \Big),
\end{equation}
where $L_{p(i)}$ and $L_i$ denote the locations of the vertices in $V_{p(i)}$ and $V_i$, and $\beta_u$ is a sign correction defined as
\begin{equation*}
\begin{aligned}
    \beta_u &:= \sign \bigg[ \P(u \sim v \mid L_{p(i)}, L_i, \sigma^*(u) = \sigma^*(v)) - \P(u \sim v \mid L_{p(i)}, L_i, \sigma^*(u) \neq \sigma^*(v)) \bigg] \\
    &= \sign \bigg[ \overline{f}_\text{in}(\|u-v\|) - \overline{f}_\text{out}(\|u-v\|) \bigg].
\end{aligned}
\end{equation*}
On the other hand, if 
\begin{equation*}
    \Big|\{u \in V_{p(i)}: \hat{\sigma}(u) = +1, u \dist v \}| < |\{u \in V_{p(i)}: \hat{\sigma}(u) = -1, u \dist v\}\Big|,
\end{equation*}
meaning that $\hat{\sigma}$ labels more vertices in $V_{p(i)}$ which $\epsilon$-distinguish $v$ as $-1$ than $+1$, then we compute
\begin{equation}
\label{eq:def_Yv_more_negative}
    Y_v := \sum_{u \in V_{p(i)} : \hat{\sigma}(u) = -1, \text{$u \dist v$}} \beta_u \Big( \1(u \sim v) - \P(u \sim v \mid L_{p(i)}, L_i) \Big),
\end{equation}
where
\begin{equation*}
\begin{aligned}
    \beta_u &:= \sign \bigg[ \P(u \sim v \mid L_{p(i)}, L_i, \sigma^*(u) \neq \sigma^*(v)) - \P(u \sim v \mid L_{p(i)}, L_i, \sigma^*(u) = \sigma^*(v)) \bigg] \\
    &= \sign \bigg[ \overline{f}_\text{out}(\|u-v\|) - \overline{f}_\text{in}(\|u-v\|) \bigg].
\end{aligned}
\end{equation*}
Then, we set the estimated label for $v$ as
\begin{equation}
\label{eq:sigma_hat_propagate}
    \hat{\sigma}(v) = \sign(Y_v)
\end{equation}
We will provide some intuition for $Y_v$. Essentially, we look at whether there is an edge between $u$ and $v$, then determine whether that gives evidence for $v$ having the same or opposite community label as $u$. Suppose for now that $\hat{\sigma}$ labels more vertices in $V_{p(i)}$ as $+1$ than $-1$, so we compute $Y_v$ using (\ref{eq:def_Yv_more_positive}). First, $\beta_u$ measures whether the likelihood of having an edge is higher when $u$ and $v$ have the same community label (in which case $\beta_u = +1$) or the opposite community label (in which case $\beta_u = -1$). If $\beta_u = +1$, then having an edge suggests that $u$ and $v$ are in the same community because the likelihood is higher. On the other hand, if $\beta_u = -1$, then having an edge would suggest $u$ and $v$ are in opposite communities. Then, $Y_v$ measures the amount of evidence we have by comparing the observed number of edges with the expected number of edges (conditioned on the vertex locations). Like with $X_v$, we will later show that the conditional expectation $\E[Y_v \mid \sigma^*(v)]$ is positive when $\sigma^*(v) = \sigma^*(u_0)$ and negative $\sigma^*(v) \neq \sigma^*(u_0)$ under certain conditions which hold with high probability. Hence, the sign of $Y_v$ can be used to infer the label of $v$.

Once we have the Phase I labeling $\hat{\sigma}$, we construct our exact estimator $\tilde{\sigma}$ by applying \texttt{Refine} on each vertex. First, we define
\begin{equation}
\label{eq:def_tau}
\begin{aligned}
    \tau(v, \sigma) := & \sum_{u \in V: \sigma(u) = +1, u \visible v} \left( \log\bigg(\frac{\overline{f}_\text{in}(\|u-v\|)}{\overline{f}_\text{out}(\|u-v\|)}\bigg) \1(u \sim v) + \log\bigg(\frac{1 - \overline{f}_\text{in}(\|u-v\|)}{1 - \overline{f}_\text{out}(\|u-v\|)}\bigg)\1(u \nsim v) \right) \\ 
    &\quad - \sum_{u \in V: \sigma(u) = -1, u \visible v} \left( \log\bigg(\frac{\overline{f}_\text{in}(\|u-v\|)}{\overline{f}_\text{out}(\|u-v\|)}\bigg) \1(u \sim v) + \log\bigg(\frac{1 - \overline{f}_\text{in}(\|u-v\|)}{1 - \overline{f}_\text{out}(\|u-v\|)}\bigg)\1(u \nsim v) \right).
\end{aligned}
\end{equation}
Then, for any given vertex $v$, we set the Phase II label as
\begin{equation}
\label{eq:exact_label}
    \tilde{\sigma}(v) = \sign(\tau(v, \hat{\sigma})).
\end{equation}
Intuitively, $\tau(v, \hat{\sigma})$ computes the likelihood ratio of $v$ having label $+1$ and $v$ having label $-1$, treating the Phase I labeling $\hat{\sigma}$ as the ground truth. Note that $\tau(v, \sigma) > 0$ if the likelihood of $\sigma(v) = +1$ is higher and $\tau(v, \sigma) < 0$ if the likelihood of $\sigma(v) = -1$ is higher. Therefore, we are just taking the label with the higher likelihood. This procedure is inspired by the genie-aided estimator, which labels a vertex $v$ knowing the labels of all other vertices. 

Now, we examine the runtime of Algorithm \ref{alg:exact_recovery}. The analysis is similar to that of Algorithm 5 in \cite{Gaudio+2024}. First, observe that the number of edges is $\Theta(n \log n)$ with high probability because the degree of each vertex is logarithmic. Then, the block visibility graph $H = (V^\dagger, E^\dagger)$ can be constructed in $O(n / \log n)$ time because $|V^\dagger| = O(n / \log n)$ since there are $n / (r \log n)$ total blocks, and $|E^\dagger| = O(n / \log n)$ as well since there are a constant number of blocks visible to any given block. If $H$ is connected, a spanning tree can be found in $O(|E^\dagger| \log |E^\dagger|)$ time via Kruskal's algorithm, which translates into $O(n \log n)$ time. \texttt{Pairwise Classify} iterates over all edges in the first block to compute $X_v$ for each vertex, which means that the runtime is $O(\log^2 n)$. \texttt{Propagate} iterates over all edges between two blocks to compute $Y_v$ for each vertex, which yields a runtime of $O(\log^2 n)$ for each block. Then, since there are $O(n / \log n)$ blocks in total, the overall runtime to complete the Phase I labeling is $O(n \log n)$. Finally, the Phase II labeling takes $O(n \log n)$ total time because for any given vertex $v$, \texttt{Refine} iterates over all vertices in the neighborhood of $v$, of which there are $O(\log n)$. Therefore, the overall runtime of Algorithm \ref{alg:exact_recovery} is $O(n \log n)$, which is linear in the number of edges.

\subsection{Proof Outline}
We now outline the proof of Theorem \ref{thm:achievability}, showing that the estimator $\tilde{\sigma}$ produced by Algorithm \ref{alg:exact_recovery} achieves exact recovery with high probability. First, we show that the $(r^d \chi \log n, \delta \log n)$-visibility graph of $G$ constructed in Line \ref{line:visibility_graph} of Algorithm \ref{alg:exact_recovery} is connected for suitably small $\chi, \delta > 0$, allowing us to use the proposed propagation scheme to label each block. Next, we show that \texttt{Pairwise Classify} labels all vertices in the initial block $V_{i_1}$ correctly and that \texttt{Propagate} labels the remaining $\delta$-occupied blocks with at most $M$ mistakes in each block, for a suitable constant $M$. Hence, we obtain that the estimator $\hat{\sigma}$ produced by Phase I of Algorithm \ref{alg:exact_recovery} achieves almost-exact recovery. Finally, we show that \texttt{Refine} labels any given vertex incorrectly with probability $o(1/n)$, which allows us to conclude that $\tilde{\sigma}$ labels all vertices correctly with probability $o(1)$ via the union bound.

The connectivity of the $(r^d \chi \log n, \delta \log n)$-visibility graph of $G$ is established by a somewhat technical reduction to the $r = 1$ case, which itself is handled in \cite{Gaudio+2024}. We can think of the block visibility graph as a coarsening of the vertex visibility graph, whose connectivity is required for exact recovery. To see this, observe that when the vertex visibility graph is disconnected, it is impossible to determine relative community labels across components, due to the symmetry in the model.

The next step is showing that all vertices in the initial block $V_{i_1}$ are labeled correctly by \texttt{Pairwise Classify} with high probability. To accomplish this, we upper-bound the probability of making a mistake on any given vertex $v \in V_{i_1}$---which are $\P(X_v \leq 0 \mid \sigma^*(v) = \sigma^*(u_0))$ and $\P(X_v \geq 0 \mid \sigma^*(v) \neq \sigma^*(u_0))$---using Hoeffding's inequality. Here, we require Assumptions \ref{eq:function_assumption_3} and \ref{eq:function_assumption_4} to ensure that the expectations $\E[X_v \mid \sigma^*(v) = \sigma^*(u_0)]$ and $\E[X_v \mid \sigma^*(v) \neq \sigma^*(u_0)]$ are bounded away from zero. We can then apply the union bound to show that all vertices are labeled correctly with high probability.

Then, we show that \texttt{Propagate} makes at most $M$ mistakes in the remaining $\delta$-occupied blocks. Suppose that we want to label the vertices $V_i$ in block $B_i$ using the labeled vertices $V_{p(i)}$ in the parent block $B_{p(i)}$. We first show that for any vertex $v \in V_i$, there are a large number of vertices in $V_{p(i)}$ which $\epsilon$-distinguish $v$ for an appropriate constant $\epsilon$ (Lemma \ref{lem:distinguishing_block_lemma}), ensuring that we have sufficiently many vertices for $Y_v$ to be close to its conditional expectation $\E[Y_v \mid \sigma^*(v)]$. Next, we prove that if $V_{p(i)}$ is labeled with at most $M$ mistakes, then $V_i$ is labeled with at most $M$ mistakes as well with probability $1 - o(1/n)$ (Proposition \ref{prop:propagation_error_vertex}). To accomplish this, we bound the probability of making a mistake on any given vertex via Hoeffding's inequality, then apply the union bound; we require Assumptions \ref{eq:function_assumption_3} and \ref{eq:function_assumption_4} here as well. Finally, since $\hat{\sigma}$ makes at most $M$ mistakes on the initial block, we can conclude that all $\delta$-occupied blocks have at most $M$ mistakes with high probability; this final step is formally shown in Theorem \ref{thm:propagation_error_total}.

Finally, we show that \texttt{Refine} labels any given vertex $v \in V$ incorrectly with probability $o(1/n)$. This requires upper-bounding the probability that $\tau(v, \sigma^*(u_0) \hat{\sigma})$ makes an error. First, we will upper-bound the probability that the genie-aided estimator $\tau(v, \sigma^*)$ comes ``close'' to making an error, showing that for any constants $\rho, \eta > 0$ we have $\P(\tau(v, \sigma^*) \geq - \rho \eta \log n \mid \sigma^*(v) = -1) \leq n^{-(I(f_\text{in}, f_\text{out}) - \rho \eta /2)}$ and $\P(\tau(v, \sigma^*) \leq \rho \eta \log n \mid \sigma^*(v) = +1) \leq n^{-(I(f_\text{in}, f_\text{out}) - \rho \eta /2)}$. Then, we show that for any $\sigma$ which differs from $\sigma^*$ by at most $\eta \log n$ vertices in the neighborhood of $v$, we have $|\tau(v, \sigma^*(u_0)\sigma) - \tau(v, \sigma^*)| \leq \rho \eta \log n$ for an appropriate constant $\rho$. Therefore, we can conclude that the probability $\tau(v, \sigma^*(u_0) \sigma)$ makes an error on $v$ for any such $\sigma$ is at most $n^{-(I(f_\text{in}, f_\text{out}) - \rho \eta /2)}.$  Now, we invoke the condition $I(f_\text{in}, f_\text{out}) > 1$ to choose $\eta > 0$ such that $I(f_\text{in}, f_\text{out}) - \rho \eta / 2 > 1$, which implies that the probability $\tau(v, \sigma^*(u_0) \sigma)$ makes an error is $o(1/n)$. Finally, we note that $\hat{\sigma}$ differs from $\sigma^*$ by at most $\eta \log n$ vertices in the neighborhood of $v$ with high probability, so the probability that $\tau(v, \sigma^*(u_0)\hat{\sigma})$ labels $v$ incorrectly is also $o(1/n)$. Applying the union bound then shows that $\tilde{\sigma}$ labels all vertices correctly (up to a global sign flip) with probability $1- o(1)$.

\section{Connectivity of the Block Visibility Graph}
\label{sec:block_size}
In this section, we show that the $(r^d \chi \log n, \delta \log n)$ visibility graph of $G$ is connected with probability $1 - o(1)$ for suitable choices of $\chi$ and $\delta$. We first record some useful concentration inequalities, which will be used throughout the rest of the paper.

\begin{lemma}[Chernoff Bound, Poisson]
\label{lem:poisson_Chernoff}
    Let $X \sim \text{Pois}(\mu)$ where $\mu > 0$. Then, for any $t > 0$,
    \begin{equation*}
        \P(X \geq \mu + t) \leq \exp\left(-\frac{t^2}{2(\mu + t)}\right)
    \end{equation*}
    and
    \begin{equation*}
        \P(X \leq \mu - t) \leq \exp \left(- (\mu - t) \log \left(1 - \frac{t}{\mu}\right) -t \right).
    \end{equation*}
\end{lemma}

\begin{lemma}[Chernoff Bound, Binomial]
\label{lem:binomial_Chernoff}
    Let $X_i: i=1, \ldots, n$ be independent Bernoulli random variables. Let $X = \sum_{i=1}^n X_i$ and $\mu = \E[X]$. Then, for any $t > 0$,
    \begin{equation*}
        \P(X \leq (1 + t)\mu) \leq \left(\frac{e^t}{(1+t)^{(1+t)}}\right)^\mu.
    \end{equation*}
\end{lemma}

Now, we turn our attention towards showing the connectivity of the visibility graph. In \cite{Gaudio+2024}, the authors proved the following connectivity result when $r = 1$. 
\begin{lemma}[Propositions 4.1 and 4.2 in \cite{Gaudio+2024}]
\label{lem:weak_connected_lemma}
Let $G \sim GSBM(\lambda, 1, n, f_\text{in}, f_\text{out}, d)$. Then, 
\begin{enumerate}
\item Suppose that $d = 1$, $\lambda > 1$, and $0 < \chi < (1 - 1/\lambda)/2$. Then, there exists a constant $\delta' > 0$ such that for any $0 < \delta < \delta' \chi$, the $(\chi \log n, \delta \log n)$-visibility graph of $G$ is connected with probability $1 - o(1)$. 
\item Suppose that $d \geq 2$, $\lambda \nu_d > 1$, and $\chi > 0$ satisfies 
\begin{equation}
\label{eq:chi_condition_without_r}
    \lambda \bigg(\nu_d \Big(1-\frac{3\sqrt{d}}{2} \chi^{1/d}\Big)^d - \chi\bigg) > 1 \quad \text{and} \quad 1-\frac{3\sqrt{d}}{2} \chi^{1/d} > 0.
\end{equation}
Then, there exists a constant $\delta''$ such that for any $0 < \delta < \delta'' \chi / \nu_d$, the $(\chi \log n, \delta \log n)$-visibility graph of $G$ is connected with probability $1 - o(1)$.
\end{enumerate}
We note that $\delta'$ and $\delta''$ are explicitly computable via  Lemma 4.4 in \cite{Gaudio+2024}.
\end{lemma}

To prove an analogous result for the $r \neq 1$ case, we first need the following extension of Lemma \ref{lem:weak_connected_lemma}.

\begin{lemma}
\label{lem:strong_connected_lemma}
Let $G \sim GSBM(\lambda, 1, n, f_\text{in}, f_\text{out}, d)$. Then,
\begin{enumerate}
\item Suppose that $d = 1$, $\lambda > 1$, and $0< \chi_0 < (1 - 1/\lambda)/2$. Then, there exists a constant $\delta' > 0$ such that for any $(\chi_0 / 2) \log n < v(n) < \chi_0 \log n$ and $\delta' v(n) / 2 < c(n) < \delta' v(n)$, the $(v(n), c(n))$-visibility graph of $G$ is connected with probability $1 - o(1)$. 
\item Suppose that $d \geq 2$, $\lambda \nu_d > 1$, and $\chi_0$ satisfies (\ref{eq:chi_condition_without_r}). Then, there exists a constant $\delta'' > 0$ such that for any $(\chi_0 / 2) \log n < v(n) < \chi_0 \log n$ and $\delta'' v(n)  / (2 \nu_d) < c(n) < \delta'' v(n) / \nu_d$, the $(v(n), c(n))$-visibility graph of $G$ is connected with probability $1 - o(1)$.
\end{enumerate}
\end{lemma}
The proof of Lemma \ref{lem:strong_connected_lemma} is virtually identical to the proof of Propositions 4.1 and 4.2 in \cite{Gaudio+2024} because the volume of each block $v(n)$ is at most $\chi_0 \log n$, which provides the existence of $\delta'$ and $\delta''$ such that the visibility graph is connected when the occupancy criterion $c(n)$ is reduced proportionally to $v(n)$. Hence, we omit the proof here.

Now, we can prove the desired connectivity result for the $r \neq 1$ case.

\begin{prop} 
\label{prop:connected_prop}
Let $G \sim GSBM(\lambda, r, n, f_\text{in}, f_\text{out}, d)$. Then,
\begin{enumerate}
\item Suppose that $d=1$, $\lambda r > 1$, and $0 < \chi_0 < (1 - 1/(\lambda r)) / 2$. Then, there exists a constant $\delta' > 0$ such that for any $\chi, \delta$ satisfying 
\begin{equation}
\label{eq:chi_delta_conditions_d=1}
    \chi_0/2 < \chi < \chi_0 \quad \text{and} \quad \delta' \chi / 2 < \delta < \delta' \chi,
\end{equation}
the $(r\chi \log n, \delta \log n)$-visibility graph of $G$ is connected with probability $1 - o(1)$. 

\item Suppose that $d \geq 2$, $\lambda r^d \nu_d > 1$, and $\chi_0$ satisfies
\begin{equation}
\label{eq:chi_condition_with_r}
\lambda r^d \bigg(\nu_d \Big(1-\frac{3\sqrt{d}}{2} \chi_0^{1/d} \Big)^d - \chi_0 \bigg) > 1 \quad \text{and} \quad 1-\frac{3\sqrt{d}}{2} \chi_0^{1/d} > 0.
\end{equation}
Then, there exists a constant $\delta'' > 0$ such that for any $\chi, \delta$ satisfying 
\begin{equation}
\label{eq:chi_delta_conditions_d>=2}
    \chi_0/2 < \chi < \chi_0 \quad \text{and} \quad \delta''\chi / (2 \nu_d) < \delta < \delta''\chi / \nu_d,
\end{equation}
the $(r^d \chi \log n, \delta \log n)$-visibility graph of $G$ is connected with probability $1 - o(1)$. 
\end{enumerate}
\end{prop}
\begin{proof}
To show that the $(r^d \chi \log n, \delta \log n)$-visibility graph of $G$ is connected with high probability, we first reduce to the $r = 1$ case by rescaling space by a factor of $1 / r$, which then allows us to apply Lemmas \ref{lem:weak_connected_lemma} and \ref{lem:strong_connected_lemma} to establish that the visibility graph is connected. This rescaling produces a graph $G_1$ with the following properties:
\begin{itemize}
    \item $G_1$ is embedded within $\calS_{d, n'}$ where $n' := n / r^d$.
    \item The intensity parameter is $\lambda' := \lambda r^d$.
    \item The visibility radius is $(\log n)^{1/d}$.
    \item The volume of each block is $\chi \log n$.
    \item The occupancy criterion is $\delta \log n$ vertices.
\end{itemize}
Since we only rescaled space, the $(\chi \log n, \delta \log n)$-visibility graph of $G_1$ is connected if and only if the $(r^d \chi \log n, \delta \log n)$-visibility graph of $G$ is connected. Thus, it is sufficient to show that the $(\chi \log n, \delta \log n)$-visibility graph of $G_1$ is connected. However, we cannot apply Lemmas \ref{lem:weak_connected_lemma} or \ref{lem:strong_connected_lemma} directly to $G_1$ directly because it is embedded in the cube $\calS_{d, n'}$ with volume $n'$ while the visibility radius is $(\log n)^{1/d}$. Instead, we will need to apply Lemmas \ref{lem:weak_connected_lemma} and \ref{lem:strong_connected_lemma} to a different graph, whose definition depends on whether $r > 1$ or $r < 1$.

Case 1: $r > 1$. Consider the graph $G_2$ with the following properties:
\begin{itemize}
    \item $G_2$ is embedded within $\calS_{d, n'}$.
    \item The intensity parameter is $\lambda'$.
    \item The visibility radius is $(\log n')^{1/d}$.
    \item The volume of each block is $\chi \log n = \chi (\log n' + d \log r)$, which we denote as $v(n')$.
    \item The occupancy criterion is $\delta \log n = \delta (\log n' + d \log r)$, which we denote as $c(n')$.
\end{itemize}

We will use Lemma \ref{lem:strong_connected_lemma} to show that the $(\chi \log n, \delta \log n)$-visibility graph of $G_2$ is connected. If $d = 1$, we have that $\lambda r > 1$ and $0 < \chi_0 < (1 - 1/(\lambda r)) / 2$, which implies that $\lambda' > 1$ and $\chi_0 < (1 - 1/\lambda')/2$. Then, we can fix $\delta' > 0$ such that Lemma \ref{lem:strong_connected_lemma} holds. Observe that $\chi_0 / 2 < \chi < \chi_0$ implies that $(\chi_0 / 2) \log n' < v(n') < \chi_0 \log n'$ because
\begin{equation*}
    v(n') = \chi (\log n' + d \log r) < \chi_0 \log n' \quad \text{for sufficiently large $n$},
\end{equation*}
and
\begin{equation*}
    v(n') = \chi (\log n' + d \log r) > \chi \log n' > (\chi_0 / 2) \log n'.
\end{equation*}
Furthermore, $\delta' \chi /2 < \delta < \delta' \chi$ implies that $\delta' v(n') / 2 < c(n') < \delta' v(n')$. Therefore, we can apply Lemma \ref{lem:strong_connected_lemma}, which shows that the $(\chi \log n, \delta \log n)$-visibility graph of $G_2$ is connected with probability $1 - o(1)$. Since $G_1$ has a larger visibility radius than $G_2$ while the other properties are the same, $G_1$ also has a connected $(\chi \log n, \delta \log n)$-visibility graph with probability $1 - o(1)$ when $d=1$.

If $d \geq 2$, we have that $\lambda r^d (\nu_d(1 - 3\sqrt{d} \chi_0^{1/d}/2)^d - \chi_0) > 1$, which implies that $\lambda' (\nu_d (1 - 3 \sqrt{d} \chi_0^{1/d} / 2)^d - \chi_0) > 1$. We also have that $1 - 3\sqrt{d}\chi_0^{1/d}/2 > 0$ by hypothesis. Thus, we can fix $\delta'' > 0$ such that Lemma \ref{lem:strong_connected_lemma} holds. Now, observe that $\chi_0 / 2 < \chi < \chi_0$ implies $(\chi_0/2) \log n' < v(n') < \chi_0 \log n'$ because
\begin{equation*}
    v(n') = \chi (\log n' + d \log r) < \chi_0 \log n' \quad \text{for sufficiently large $n$},
\end{equation*}
and
\begin{equation*}
    v(n') = \chi (\log n' + d \log r) > \chi \log n' > (\chi_0 / 2) \log n'.
\end{equation*}
In addition, $\delta'' \chi / (2 \nu_d) < \delta < \delta'' \chi / \nu_d$ implies that $\delta'' v(n') / (2 \nu_d) < c(n') < \delta '' v(n') / \nu_d$. Therefore, we can apply Lemma \ref{lem:strong_connected_lemma}, which shows that the $(\chi \log n, \delta \log n)$-visibility graph of $G_2$ is connected with probability $1 - o(1)$. Since $G_1$ has a larger visibility radius than $G_2$ while the other properties are the same, $G_1$ also has a connected $(\chi \log n, \delta \log n)$-visibility graph with probability $1 - o(1)$ when $d \geq 2$.

Case 2: $r < 1$. In this case, we will divide $\calS_{d, n'}$ into overlapping cubes of volume $n$, which we call tiles. We will then apply Lemma \ref{lem:weak_connected_lemma} to show that each tile has a connected $(\chi \log n, \delta \log n)$-visibility graph, and also show that each overlap has a non-empty block. This allows us to conclude that $G_1$ has a connected $(\chi \log n, \delta \log n)$-visibility graph.

If $d = 1$, we divide $\calS_{d, n'}$ into overlapping intervals of length $n$, which produces a total of $\floor{1/r} + 1$ tiles and $\floor{1/r}$ overlaps (the leftmost and rightmost tiles do not overlap). If $d \geq 2$, we divide $\calS_{d, n'}$ into overlapping intervals of length $n$ in each dimension and form tiles by taking one interval from each dimension, which produces a total of $(\floor{1/r} + 1)^d$ tiles. We can uniquely label each tile with a $d$-tuple $\alpha$ using the following scheme: In each dimension, we can assign each interval an index from $1$ to $\floor{1/r} + 1$ going left-to-right. Then, we set the $i^\text{th}$ element of $\alpha$, which we denote $\alpha_i$, as the index of the interval in the $i^\text{th}$ dimension which forms the tile. Then, we call two tiles with labels $\alpha$ and $\beta$ adjacent if there exists a dimension $j$ such that $\alpha_i = \beta_i$ for $i \neq j$ while $|\alpha_j - \beta_j| = 1$. In other words, two tiles are adjacent if they are formed from the same interval in $d - 1$ dimensions and distinct, overlapping intervals in the remaining dimension. Now, observe that between any two adjacent tiles, there is an overlap of length $n^{1/d}(\floor{1/r} + 1 - 1/r) / (\floor{1/r})$ in one dimension---this is because the length of $\floor{1/r} + 1$ tiles placed side-by-side is $n^{1/d} (\floor{1/r} + 1)$, but they overlap slightly so that their total length is $n^{1/d}/r$. Therefore, there is an overlap of volume $n(\floor{1/r} + 1 - 1/r) / (\floor{1/r})$ between any two adjacent tiles.

Now, let $M := (\floor{1/r} + 1)^d$ be the total number of tiles. Suppose that we number the tiles $T_1, T_2, \ldots, T_M$ such that tiles $T_i$ and $T_{i+1}$ are adjacent. We denote the region of overlap between $T_i$ and $T_{i+1}$ as $O_i$. Formally, $O_i := T_i \cap T_{i+1}$. Then, observe that the visibility graph of $G_1$ is connected if each tile is has a connected visibility graph and each $O_i$ has at least one occupied block. Therefore, we will show that each tile is connected  and each $O_i$ has at least one occupied block with probability $1 - o(1)$. Let $\calE_1$ be the event that each tile is connected and $\calE_2$ be the event that each $O_i$ has at least one occupied block.

We first show that $\calE_1$ occurs with probability $1 - o(1)$. Observe that each tile has the following properties:
\begin{itemize}
    \item The volume of each tile is $n$.
    \item The intensity parameter is $\lambda'$ .
    \item The visibility radius is $\log n$.
    \item The volume of each block is $\chi \log n$.
    \item The occupancy criterion is $\delta \log n$.
\end{itemize}

If $d = 1$, we have that $\lambda r > 1$ and $0 < \chi_0 < (1 - 1/(\lambda r)) / 2$, which implies that $\lambda' > 1$ and $\chi_0 < (1 - 1/\lambda') / 2$. Then, since $\chi_0 / 2 < \chi < \chi_0$, we obtain that $0 < \chi < (1 - 1/(\lambda r)) / 2$. Hence, we can fix $\delta'$ such that Lemma \ref{lem:weak_connected_lemma} holds. Then, since $0 < \delta < \delta' \chi$, we can apply Lemma \ref{lem:weak_connected_lemma} to obtain that the $(\chi \log n, \delta \log n)$-visibility graph of any given tile is connected with probability $1- o(1)$. Finally, since there are $(\floor{1/r} + 1)^d$ tiles (which is constant in $n$), the union bound gives us $\P(\calE_1) = 1 - o(1)$ when $d = 1$. 

If $d \geq 2$, we have that $\lambda r^d (\nu_d(1 - 3\sqrt{d} \chi_0^{1/d}/2)^d - \chi_0) > 1$, which implies that $\lambda' (\nu_d (1 - 3 \sqrt{d} \chi_0^{1/d} / 2)^d - \chi_0) > 1$. We also have that $1 - 3\sqrt{d}\chi_0^{1/d}/2 > 0$ by hypothesis. Then, since $\chi_0/2 < \chi < \chi_0$, we obtain that $\lambda' (\nu_d (1 - 3 \sqrt{d} \chi^{1/d} / 2)^d - \chi) > 1$ and $1 - 3\sqrt{d}\chi^{1/d}/2 > 0$. Hence, we can fix $\delta''$ such that Lemma \ref{lem:weak_connected_lemma} holds. Then, since $0 < \delta < \delta'' \chi / \nu_d$, we can apply Lemma \ref{lem:weak_connected_lemma} to obtain that the $(\chi \log n, \delta \log n)$-visibility graph of any given tile is connected with probability $1- o(1)$. Finally, since there are $(\floor{1/r} + 1)^d$ tiles, the union bound gives us $\P(\calE_1) = 1 - o(1)$ when $d \geq 2$ as well. 

Now, we show that any given $O_i$ contains at least one occupied block with probability $1 - o(1)$. For any given block $B_i$, define $U_i$ as the set of blocks visible to $B_i$. That is,
\begin{equation}
\label{eq:def_Ui}
    U_i = \bigcup_{j: j \neq i, B_j \visible B_j} B_j.
\end{equation}
Then, let $K$ be the number of blocks in $U_i$. Since $\text{Vol}(U_i) = \Theta(\log n)$ but $\text{Vol}(O_i) = \Theta(n)$, we see that $O_i$ contains more than $K$ blocks for sufficiently large $n$. A corollary of Lemma 4.6 in \cite{Gaudio+2024} is that for any given set of $K$ blocks $B_{j_1}, B_{j_2}, \ldots, B_{j_K}$, there exists $\epsilon > 0$ such that
\begin{equation*}
    \P(\bigcap_{k=1}^K \{|V_{jk}| \leq \delta \log n \}) \leq n^{-(1+\epsilon)}.
\end{equation*}
That is, for any given set of $K$ blocks, the probability that all $K$ blocks are $\delta$-unoccupied is $o(1)$. Then, since $O_i$ has more than $K$ blocks for sufficiently large $n$,
\begin{equation*}
    \P(\bigcap_{B_i \subseteq O_i} \{|V_i| \leq \delta \log n\}) \leq n^{-(1+\epsilon)}.
\end{equation*}
That is, the probability that all blocks in $O_i$ are $\delta$-unoccupied is $o(1)$. Finally, since there are $(\floor{1/r} + 1)^d - 1$ overlaps $O_i$ (which is constant in $n$), we can use the union bound to obtain $\P(\calE_2) = 1 - o(1)$. 

Therefore, we have that
\begin{equation*}
    \P(\calE_1 \cap \calE_2) = 
    1 - \P(\calE_1^c \cup \calE_2^c) \geq 
    1 - \P(\calE_1^c) - \P(\calE_2^c) 
    = 1 - o(1)
\end{equation*}
which shows that the $(\chi \log n, \delta \log n)$-visibility graph of $G_1$ is connected with probability $1 - o(1)$. 
\end{proof}

\section{Labeling an Initial Block}
\label{sec:label_initial}
In this section, we will show that Algorithm \ref{alg:exact_recovery} produces the correct labeling for all vertices in the initial block $V_{i_1}$ with high probability.

\begin{prop}
\label{prop:initial_block_labeling}
Let $\hat{\sigma}$ be the estimated labels for the initial block $V_{i_1}$ produced by Line \ref{line:pairwise_classify} of Algorithm \ref{alg:exact_recovery}. Then, for any positive constants $\delta < \Delta$,
\begin{equation*}
    \mathbb{P}\bigg( \bigcup_{v \in V_{i_1} \setminus \{u_0\}} \big\{ \hat{\sigma}(v) = \sigma^*(v) \sigma^*(u_0) \big\} \: \Big| \: \delta \log n < |V_{i_1}| < \Delta \log n \bigg) = 1 - o(1).
\end{equation*}
That is, conditioned on $\delta \log n < |V_{i_1}| < \Delta \log n$, $\hat{\sigma}$ is correct for all vertices in $V_{i_1}$ (up to a global sign flip) with probability $1 - o(1)$. 
\end{prop}
\begin{proof}
We need to upper-bound the probabilities of the two ways of making an error, which are 
\begin{equation*}
    \mathbb{P}(X_v \leq 0 \mid \sigma^*(v) = \sigma^*(u_0)) \quad \text{ and } \quad  \mathbb{P}(X_v  \geq 0 \mid \sigma^*(v) \neq \sigma^*(u_0)).
\end{equation*}
We start by upper-bounding the first probability. Let $L_{i_1}$ denote the locations of the vertices in $V_{i_1}$. First, we compute the expectation $\E\left[X_v \mid \sigma^*(v) = \sigma^*(u_0), |V_{i_1}| = n_1, L_{i_1} \right]$, which yields
\begin{equation*}
\begin{aligned}
    & \E\left[X_v \mid \sigma^*(v) = \sigma^*(u_0), |V_{i_1}| = n_1, L_{i_1} \right] \\
    &= \E \bigg[ \sum_{u \in V_{i_1} \setminus\{u_0, v\}} \alpha_u \Big(\1\{u \sim u_0, u \sim v\} - \mathbb{P}(u \sim u_0, u \sim v \mid L_{i_1} )\Big) \: \Big| \: \sigma^*(v) = \sigma^*(u_0), |V_{i_1}| = n_1, L_{i_1} \bigg] \\
    &= \sum_{u \in V_{i_1} \setminus\{u_0, v\}} \alpha_u \Big( \E\left[ \1\{u \sim u_0, u \sim v\} \mid \sigma^*(v) = \sigma^*(u_0), L_{i_1} \right] - \P(u \sim u_0, u \sim v \mid L_{i_1}) \Big) \\
    &= \sum_{u \in V_{i_1} \setminus\{u_0, v\}} \alpha_u \Big( \P(u \sim u_0, u\sim v \mid \sigma^*(v) = \sigma^*(u_0), L_{i_1}) - \P(u \sim u_0, u \sim v \mid L_{i_1}) \Big) \\
    &= \sum_{u \in V_{i_1} \setminus\{u_0, v\}} \frac{1}{2}\alpha_u \Big( \P(u \sim u_0, u\sim v \mid \sigma^*(v) = \sigma^*(u_0), L_{i_1}) - \P(u \sim u_0, u \sim v \mid \sigma^*(v) \neq \sigma^*(u_0), L_{i_1}) \Big) \\
    &= \sum_{u \in V_{i_1} \setminus\{u_0, v\}} \frac{1}{2} \Big| \P(u \sim u_0, u\sim v \mid \sigma^*(v) = \sigma^*(u_0), L_{i_1}) - \P(u \sim u_0, u \sim v \mid \sigma^*(v) \neq \sigma^*(u_0), L_{i_1}) \Big|.
\end{aligned} 
\end{equation*}
Observe that
\begin{equation*}
    \P(u \sim u_0, u\sim v \mid \sigma^*(v) = \sigma^*(u_0), L_{i_1}) = \frac{1}{2} \Big( \overline{f}_\text{in}(\|u-u_0\|) \overline{f}_\text{in}(\|u-v\|) + \overline{f}_\text{out}(\|u-u_0\|) \overline{f}_\text{out}(\|u-v\|) \Big)
\end{equation*}
and
\begin{equation*}
    \P(u \sim u_0, u\sim v \mid \sigma^*(v) \neq \sigma^*(u_0), L_{i_1}) = \frac{1}{2} \Big( \overline{f}_\text{in}(\|u-u_0\|) \overline{f}_\text{out}(\|u-v\|) + \overline{f}_\text{out}(\|u-u_0\|) \overline{f}_\text{in}(\|u-v\|) \Big).
\end{equation*}
Therefore, we obtain
\begin{equation*}
\begin{aligned}
     & \E\left[X_v \mid  \sigma^*(v) = \sigma^*(u_0), |V_{i_1}| = n_1, L_{i_1} \right] = \\
     &= \sum_{u \in V_{i_1} \setminus\{u_0, v\}} \frac{1}{4} \Big| \overline{f}_\text{in}(\|u-u_0\|) \overline{f}_\text{in}(\|u-v\|) + \overline{f}_\text{out}(\|u-u_0\|) \overline{f}_\text{out}(\|u-v\|) \\
     & \qquad \qquad \qquad \qquad - \overline{f}_\text{in}(\|u-u_0\|) \overline{f}_\text{out}(\|u-v\|) + \overline{f}_\text{out}(\|u-u_0\|) \overline{f}_\text{in}(\|u-v\|) \Big| \\
     &=\sum_{u \in V_{i_1} \setminus\{u_0, v\}} \frac{1}{4} \Big| \overline{f}_\text{in}(\|u-u_0\|) - \overline{f}_\text{out}(\|u-u_0\|) \Big| \Big| \overline{f}_\text{in}(\|u-v\|) - \overline{f}_\text{out}(\|u-v\|) \Big|.
\end{aligned}
\end{equation*}
Letting $\overline{f}_\text{diff}(t) := \overline{f}_\text{in}(t) - \overline{f}_\text{out}(t)$, we have
\begin{equation*}
    \E\left[X_v \mid  \sigma^*(v) = \sigma^*(u_0), |V_{i_1}| = n_1, L_{i_1} \right] = \sum_{u \in V_{i_1} \setminus\{u_0, v\}} \frac{1}{4} \Big| \overline{f}_\text{diff}(\|u-u_0\|) \Big| \Big| \overline{f}_\text{diff}(\|u-v\|) \Big|.
\end{equation*}

Now, fix $\epsilon > 0$. For any given vertex $w$ in block $B_{i_1}$, we define the event
\begin{equation*}
    \calA_\epsilon(w) := \Big\{ | \overline{f}_\text{diff}(\|w-u_0\|)| > \epsilon \Big\} \bigcap \Big\{| \overline{f}_\text{diff}(\|w-v\|)| > \epsilon\Big\}.
\end{equation*}
Then, let $\calE$ be the event that $\calA_\epsilon(u)$ holds for at least half of the vertices $u \in V_1 \setminus \{u_0, v\}$. That is,
\begin{equation*}
    \calE = \bigg\{ \sum_{u \in V_{i_1} \setminus \{u_0, v\}} \1\{\calA_\epsilon(u)\} \geq \frac{1}{2}(|V_{i_1}|-2) \bigg\}.
\end{equation*}
Observe that conditioned on the event $\calE$, we have
\begin{equation*}
\begin{aligned}
    \E\left[X_v \mid  \sigma^*(v) = \sigma^*(u_0), |V_{i_1}| = n_1, \calE \right] > \frac{(n_1-2) \epsilon^2}{8}.
\end{aligned}
\end{equation*}
Then, by the law of total probability, we obtain
\begin{equation*}
\begin{aligned}
    \E\left[X_v \mid  \sigma^*(v) = \sigma^*(u_0), |V_{i_1}| = n_1 \right] > \frac{(n_1-2) \epsilon^2}{8} \P(\calE \mid |V_{i_1}| = n_1).
\end{aligned}
\end{equation*}
Now, we compute a lower bound for $\P(\calE \mid |V_{i_1}| = n_1)$. Let $w$ be a vertex placed uniformly at random within block $B_{i_1}$. Recall Assumption \ref{eq:function_assumption_4}, which states that
\begin{equation*}
\Big\{ 0 < t < r: |f_\text{diff}(t)|  \leq \epsilon \Big\} \subseteq \bigcup_{i=1}^m \Big[ t_i-\frac{\gamma(\epsilon)}{2}, t_i + \frac{\gamma(\epsilon)}{2} \Big]
\end{equation*}
where $\gamma(\epsilon) \to 0$ as $\epsilon \to 0$. Thus, the region for which $|\overline{f}_\text{diff}(\|u - u_0\|) | \leq \epsilon$ is the union of $m$ annuli centered at $u$ with inner radii $(t_i - \gamma(\epsilon)/2) (\log n)^{1/d}$ and $(t_i + \gamma(\epsilon)/2) (\log n)^{1/d}$. This region has volume at most
\begin{equation*}
    \sum_{i=1}^m \nu_d \Big( (t_i + \gamma(\epsilon)/2)^d - (t_i - \gamma(\epsilon)/2)^d \Big) \log n
\end{equation*}
which implies that
\begin{equation*}
\begin{aligned}
    \P\Big(| \overline{f}_\text{diff}(\|w-u_0\|) | < \epsilon \Big) & \leq \frac{1}{r^d \chi \log n} \sum_{i=1}^m \nu_d \Big( (t_i + \gamma(\epsilon)/2)^d - (t_i - \gamma(\epsilon)/2)^d \Big) \log n \\
    & = \frac{1}{r^d \chi} \sum_{i=1}^m \nu_d \Big( (t_i + \gamma(\epsilon)/2)^d - (t_i - \gamma(\epsilon)/2)^d \Big)
\end{aligned}
\end{equation*}
where first inequality holds because the block has volume $r^d \chi \log n$. Now, let $\overline{t} = \max_i t_i$ and observe that $(t_i + \gamma(\epsilon)/2)^d - (t_i - \gamma(\epsilon)/2)^d < (\overline{t} + \gamma(\epsilon)/2)^d - (\overline{t} - \gamma(\epsilon)/2)^d$ for all $i = 1, 2, \ldots, m$. Therefore,
\begin{equation*}
    \P\Big(| \overline{f}_\text{diff}(\|w-u_0\|) | < \epsilon \Big) < \frac{1}{r^d \chi} m \nu_d \Big( (\overline{t} + \gamma(\epsilon)/2)^d - (\overline{t} - \gamma(\epsilon)/2)^d \Big).
\end{equation*}
A similar computation shows that
\begin{equation*}
    \P\Big(| \overline{f}_\text{diff}(\|w-v\|) | < \epsilon \Big) < \frac{1}{r^d \chi} m \nu_d \Big( (t + \gamma(\epsilon)/2)^d - (t - \gamma(\epsilon)/2)^d \Big).
\end{equation*}
Hence, we obtain that
\begin{equation*}
     \P (\calA_\epsilon(w)) > 1 -  \frac{2}{r^d \chi} m \nu_d \Big( (t + \gamma(\epsilon)/2)^d - (t - \gamma(\epsilon)/2)^d \Big)
\end{equation*}
via the union bound. Now, let $q_\epsilon := \P(A_\epsilon(w))$ and $X$ be the number of vertices $u \in V_1 \setminus \{u_0, v\}$ for which $\calA_\epsilon(u)$ occurs. Then, observe that event $\calE$ occurs if and only if $X > (n_1 - 2)/2$. Furthermore, conditioned on $|V_{i_1}| = n_1$ we have $X \sim \text{Binom}(n_1 - 2, q_\epsilon)$. Therefore, we obtain that
\begin{equation*}
    \P(\calE \mid |V_{i_1}| = n_1) = \P \left( \text{Binom}(n_1 - 2, q_\epsilon) > \frac{1}{2}(n_1-2) \right).
\end{equation*}
Since $\gamma(\epsilon) \to 0$ as $\epsilon \to 0$, we also have that $q_\epsilon \to 1$ as $\epsilon \to 0$. Consequently, we can take $\epsilon$ such that $q_\epsilon > 3/4$, which shows that
\begin{equation*}
    \P(\calE \mid |V_{i_1}| = n_1) \geq \P \bigg (\text{Binom}\left(n_1 - 2, \frac{3}{4}\right) > \frac{1}{2}(n_1 - 2) \bigg) > \frac{1}{2}
\end{equation*}
where the last inequality follows from direct calculation. Therefore, $\P(\calE \mid |V_{i_1}| = n_1) \geq 1/2$, which yields
\begin{equation*}
\begin{aligned}
    \E\left[X_v \mid  \sigma^*(v) = \sigma^*(u_0), |V_{i_1}| = n_1 \right] > \frac{(n_1-2) \epsilon^2}{16}.
\end{aligned}
\end{equation*}

Now, we bound the probability of an error. If $\sigma^*(v) = \sigma^*(u_0)$, an error occurs if $X_v \leq 0$. Hence, the probability of making an error is $P( X_v \leq 0 \mid \sigma^*(u_0) = \sigma^*(v) , |V_{i_1}|=n_1)$. Since $X_v$ is the sum of independent random variables bounded between $-1$ and $1$, we can apply Hoeffding's inequality to obtain 
\begin{equation*}
\begin{aligned}
    \P( X_v \leq 0 \mid \sigma^*(u_0) = \sigma^*(v) , |V_{i_1}|=n_1) & \leq \exp\left(-\frac{2 ((n_1 - 2)\epsilon^2/16)^2}{4(n_1 - 2)}\right) \\
    &= \exp\left(- \frac{(n_1-2)\epsilon^4}{512}\right). \\
\end{aligned}
\end{equation*}
which gives us an upper bound for the error probability when $\sigma^*(u_0) = \sigma^*(v)$. Now, we examine $\mathbb{P}(X_v  \geq 0 \mid \sigma^*(v) \neq \sigma^*(u_0))$, the probability of making an error when $\sigma^*(u_0) \neq \sigma^*(v)$. We use a virtually identical approach. First, we compute the expectation $\E\left[X_v \mid \sigma^*(v) \neq \sigma^*(u_0), |V_{i_1}| = n_1 \right]$, which yields
\begin{equation*}
    \E\left[X_v \mid \sigma^*(v) \neq \sigma^*(u_0), |V_{i_1}| = n_1 \right] < - \frac{(n_1 - 2)\epsilon^2}{16}.
\end{equation*}
Then, we use Hoeffding's inequality to upper bound the probability of an error when $\sigma^*(v) \neq \sigma^*(u_0$), which shows that
\begin{equation*}
    \P(X_v \geq 0 \mid \sigma^*(u_0) \neq \sigma^*(v), |V_{i_1}| = n_1) \leq \exp\left(- \frac{(n_1-2)\epsilon^4}{512}\right).
\end{equation*}
Therefore, we also obtain an upper bound for the error probability when $\sigma^*(u_0) \neq \sigma^*(v)$. Now, we can use the law of total probability to compute the overall probability of making an error when labeling vertex $v$. Observe that
\begin{equation*}
\begin{aligned}
    \P\left(\hat{\sigma}(v) \neq \sigma^*(v) \sigma^*(u_0) \mid |V_{i_1}| = n_1 \right) &= \frac{1}{2}  \P( X_v \leq 0 \mid \sigma^*(u_0) = \sigma^*(v) , |V_{i_1}|=n_1) \\
    & \qquad + \frac{1}{2} \P(X_v \geq 0 \mid \sigma^*(u_0) \neq \sigma^*(v), |V_{i_1}| = n_1) \\
    &\leq \exp\left(- \frac{(n_1-2)\epsilon^4}{512}\right).
\end{aligned}
\end{equation*}
If $\delta \log n < |V_{i_1}| < \Delta \log n$, then the probability of making an error when labeling $v$ becomes
\begin{equation*}
\begin{aligned}
    \P\left(\hat{\sigma}(v) \neq \sigma^*(v) \sigma^*(u_0) \mid \delta \log n < |V_{i_1}| < \Delta \log n \right) &\leq \exp\left(-\frac{(\delta \log n - 2) \epsilon^4}{512}\right) \\
    &= \exp\left(\frac{\epsilon^4}{256}\right) \exp\left(-\frac{\delta \epsilon^4}{512} \log n\right) \\
    &= \exp\left(\frac{\epsilon^4}{256}\right) n^{-\delta \epsilon^4 / 512} \\
    &= \eta_1 n^{-c_1}
\end{aligned}
\end{equation*}
where $\eta_1 := \exp(\epsilon^4/256)$ and $c_1 := \delta \epsilon^4 / 512$ are constants in $n$. Finally, we apply the union bound to upper-bound the probability that some $v \in V_1$ is incorrectly labeled relative to $u_0$, which yields
\begin{equation*}
    \mathbb{P}\bigg( \bigcup_{v \in V_{i_1} \setminus \{u_0\}} \big\{ \hat{\sigma}(v) \neq \sigma^*(v) \sigma^*(u_0) \big\} \: \bigg| \: \delta \log n < |V_1| < \Delta \log n \bigg) \leq \eta_1 n^{-c_1} \Delta \log n = o(1).
\end{equation*}
Therefore, we have shown that $\hat{\sigma}$ correctly labels all vertices in $V_{i_1}$ (up to a global sign flip) with probability $1 - o(1)$.
\end{proof}

\section{Propagate}
\label{sec:propagate}
In this section, we will show that the Phase I labeling $\hat{\sigma}$ produced by Algorithm \ref{alg:exact_recovery} makes at most a constant number of mistakes in each $\delta$-occupied block. First, we introduce the following definition.

\begin{definition}[Distinguishing Block]
    Let $U_i$ be defined as in (\ref{eq:def_Ui}). We call a block $B_i$ $(\alpha, \epsilon)$-distinguishing if for every vertex $v \in U_i$, there are at least $\alpha \log n$ vertices in $V_i$ that $\epsilon$-distinguish $v$. We remark that for two vertices $v_1 \neq v_2$, the set of $\alpha \log n$ points which $\epsilon$-distinguish them can be different.
\end{definition}

We will show that all blocks are $(\alpha, \epsilon)$-distinguishing with high probability for suitable values of $\alpha$ and $\epsilon$. The following result will be helpful in our calculations.

\begin{lemma}
\label{lem:limit_calculation_lemma}
    For any constants $t \geq 0$, $k \in \N$ and $d \in \N$, 
    \begin{equation}
    \label{eq:limit_calculation_lemma}
    \lim_{x \to 0} \frac{1}{x^k} \bigg(  (t + x)^d - (t - x)^d \bigg)^{k+1} = 0.
    \end{equation}
\end{lemma}
\begin{proof}
    We will use induction on $k$. If $k = 0$, then 
    \begin{equation*}
    \lim_{x \to 0} \frac{1}{x^k} \bigg(  (t + x)^d - (t - x)^d \bigg)^{k+1} = \lim_{x \to 0} (t+x)^d - (t-x)^d = 0.
    \end{equation*}
    Thus, \eqref{eq:limit_calculation_lemma} holds when $k = 0$. Now, suppose that \eqref{eq:limit_calculation_lemma} holds for some $k \in \mathbb{N}$. Using L'H\^opital's rule, we obtain that
    \begin{equation*}
    \begin{aligned}
        \lim_{x \to 0} \frac{1}{x^{k+1}} \bigg( (t + x)^d - (t - x)^d \bigg)^{k+2} &= \lim_{x \to 0} \frac{d(k+2)}{(k+1)x^{k}} \bigg( (t + x)^d - (t - x)^d \bigg)^{k+1} \bigg( (t+x)^{d-1} + (t-x)^{d-1} \bigg) \\
        &= \lim_{x \to 0} \frac{1}{(k+1)x^k}\bigg( (t + x)^d - (t - x)^d \bigg)^{k+1} \lim_{x \to 0} d(k+2)\bigg( (t+x)^{d-1} + (t-x)^{d-1} \bigg).
    \end{aligned}
    \end{equation*}
    Observe that the first limit is zero by the inductive hypothesis, while the second limit approaches a finite constant. Hence, (\ref{eq:limit_calculation_lemma}) also holds for $k+1$. Therefore, we can conclude that (\ref{eq:limit_calculation_lemma}) holds for all $k \in \N$.
\end{proof}

Now, we can show that all blocks are $(\alpha, \epsilon)$-distinguishing with high probability for suitable values of $\alpha$ and $\epsilon$.

\begin{lemma}
\label{lem:distinguishing_block_lemma}
    Fix $\delta > 0$. Then, there exists $\epsilon > 0$ such that any given $\delta$-occupied block is $(\delta / (2(d+1)), \epsilon)$-distinguishing with probability $1 - o(1/n)$. Consequently, all $\delta$-occupied blocks are $(\delta / (2(d+1)), \epsilon)$-distinguishing with probability $1 - o(1)$.
\end{lemma}
\begin{proof}
    Let $S = \{u_1, u_2, \ldots, u_{d+1}\}$ be a set of $d + 1$ points chosen uniformly at random from a $\delta$-occupied block $B_i$, and let $\calE(S)$ be the event that for all $v \in U_i$, there is at least one point in $S$ which $\epsilon$-distinguishes $v$. First, we will show that for any constant $\kappa > 0$, there exists $\epsilon > 0$ such that $\P(\calE(S)) \geq 1 - \kappa$. 

    To accomplish this, we fix $\kappa > 0$ and recall Assumption \ref{eq:function_assumption_4}, which states that
    \begin{equation*}
    \Big\{ 0 < t < r: |f_\text{diff}(t)| \leq \epsilon \Big\} \subseteq \bigcup_{k=1}^m \Big[ t_k-\frac{\gamma(\epsilon)}{2}, t_k + \frac{\gamma(\epsilon)}{2} \Big]
    \end{equation*}
    where $\gamma(\epsilon) \to 0$ as $\epsilon \to 0$. Thus, we see that the region not $\epsilon$-distinguished by a point $u_\ell \in S$, which we denote $R(u_\ell)$, is contained within a set of $m$ annuli centered at $u_\ell$ with inner radii $(t_k - \gamma(\epsilon)/2)(\log n)^{1/d}$ and outer radii $(t_k + \gamma(\epsilon)/2)(\log n)^{1/d}$ for $k = 1, 2, \ldots, m$. That is,
    \begin{equation*}
        R(u_\ell) \subseteq \bigcup_{k=1}^m \left\{ x \in U_i: \Big(t_k - \frac{\gamma(\epsilon)}{2}\Big) (\log n)^{1/d} \leq \|x - u_\ell\| \leq \Big(t_k + \frac{\gamma(\epsilon)}{2}\Big) (\log n)^{1/d} \right\}.
    \end{equation*}
    Hence, the region not $\epsilon$-distinguished by any point in $S$, which we denote $R(S)$, is contained within the union of $R(u_\ell)$ for $\ell = 1, 2, \ldots, d+1$. That is,
    \begin{equation*}
        R(S) \subseteq \bigcap_{\ell = 1}^{d+1} \bigcup_{k=1}^m \left\{ x \in U_i: (t_k - \frac{\gamma_k(\epsilon)}{2}) (\log n)^{1/d} \leq \|x - u_\ell\| \leq (t_k + \frac{\gamma_k(\epsilon)}{2}) (\log n)^{1/d} \right\}.
    \end{equation*}
    Now, observe that the event $\calE(S)$ occurs if and only if $R(S) = \varnothing$. Therefore, to show that $\P(\calE(S)) \geq 1 - \kappa$, we can equivalently show that $\P(R(S) \neq \varnothing) < \kappa$.
    
    To do this, we will divide $U_i$ into hypercubes of volume $s \log n$ (where $s$ is a constant that we will determine later) using the following procedure: First, observe that $U_i$ is contained within a sphere of radius $r (\log n)^{1/d}$ because every point in $U_i$ must be visible to the center of $B_i$. Therefore, we can enclose $U_i$ in a hypercube $U_i'$ of side length $2r (\log n)^{1/d}$. We then divide $U_i'$ into hypercubes of volume $s \log n$, which we call $C_1, C_2, \ldots, C_{(2r)^d/s}$.

    To find an upper bound for $\P(R(S) \neq \varnothing)$, we will first upper-bound $\P(R(S) \cap C_j \neq \varnothing)$ and then apply the union bound. For any fixed set $C_j$, define the set
    \begin{equation*}
        A_j := \bigcup_{k=1}^m \left\{ y \in B_i : \exists w \in C_j \text{ such that } \Big(t_k - \frac{\gamma(\epsilon)}{2}\Big) (\log n)^{1/d} \leq \|y - w\| \leq \Big(t_k + \frac{\gamma(\epsilon)}{2}\Big) (\log n)^{1/d}  \right\}
    \end{equation*}
    which is the set of points $y \in B_i$ which do not $\epsilon$-distinguish $C_j$. We claim that $\P(R(S) \cap C_j \neq \varnothing) \leq \P(u_1, u_2, \ldots, u_{d+1} \in A_j)$. To see this, observe that if $\{R(S) \cap C_j \neq \varnothing \}$, then there exists a point $w \in C_j$ which is not $\epsilon$-distinguished by any point in $S$. Thus, no points in $S$ can $\epsilon$-distinguish $C_j$, which shows that $u_1, u_2, \ldots, u_{d+1} \in A_j$. 

    Therefore, our task becomes upper-bounding $\P(u_1, u_2, \ldots, u_{d+1} \in A_j)$. Observe that $C_j$ is contained within a sphere $C_j'$ of radius $\sqrt{d}(s \log n)^{1/d}/2$. Then, we analogously define $A_j'$ as the set
    \begin{equation*}
        A_j' := \bigcup_{k=1}^m \left\{ y \in B_i : \exists w \in C_j' \text{ such that } \Big(t_k - \frac{\gamma(\epsilon)}{2}\Big) (\log n)^{1/d} \leq \|y - w\| \leq \Big(t_k + \frac{\gamma(\epsilon)}{2}\Big) (\log n)^{1/d}  \right\}
    \end{equation*}
    which is the set of points $y \in B_i$ which do not $\epsilon$-distinguish $C_j'$. If we let $w_0$ be the center of $C_j'$, then we can concisely express $A_j'$ as
    \begin{equation*}
        A_j' = \bigcup_{k=1}^m  \left\{ y \in B_i: \bigg(t_k - \frac{\gamma(\epsilon)}{2} - \frac{\sqrt{d}}{2}s^{1/d}\bigg) (\log n)^{1/d} \leq \|y - w_0 \| \leq \bigg(t_k + \frac{\gamma(\epsilon)}{2} + \frac{\sqrt{d}}{2}s^{1/d}\bigg) (\log n)^{1/d} \right\}
    \end{equation*}
    which consists of $m$ annuli with inner radii $(t_k - \gamma(\epsilon)/2 - s^{1/d}\sqrt{d}/2) (\log n)^{1/d}$ and outer radii $(t_k + \gamma(\epsilon)/2 + s^{1/d}\sqrt{d}/2) (\log n)^{1/d}$ for $k = 1, 2, \ldots, m$. Now, note that $A_j \subset A_j'$ because any point which does not $\epsilon$-distinguish $C_j$ also does not $\epsilon$-distinguish the larger set $C_j'$. Therefore, we have that $\P(u_1, u_2, \ldots, u_{d+1} \in A_j) \leq \P(u_1, u_2, \ldots, u_{d+1} \in A_j')$. 
    
    We can derive an explicit upper bound for the latter probability. Since $A_j'$ consists of $m$ annuli,
    \begin{equation*}
        \text{Vol}(A_j') \leq \nu_d \log n \sum_{k=1}^m  \bigg(t_k + \frac{\gamma(\epsilon)}{2} + \frac{\sqrt{d}}{2}s^{1/d}\bigg)^d - \bigg(t_k - \frac{\gamma(\epsilon)}{2} - \frac{\sqrt{d}}{2}s^{1/d}\bigg)^d.
    \end{equation*}
    Letting $\overline{t} := \max_k t_k$, we have that
    \begin{equation*}
        \text{Vol}(A_j') \leq m \nu_d \log n \Bigg( \bigg(\overline{t} + \frac{\gamma(\epsilon)}{2} + \frac{\sqrt{d}}{2}s^{1/d}\bigg)^d - \bigg(\overline{t} - \frac{\gamma(\epsilon)}{2} - \frac{\sqrt{d}}{2}s^{1/d}\bigg)^d \Bigg).
    \end{equation*}
    Then, we use the fact that $u_1, u_2, \ldots, u_{d+1}$ are distributed independently and uniformly at random within a block $B_i$ of volume $r^d \chi \log n$ to obtain that
    \begin{equation*}
        \P(u_1, u_2, \ldots, u_{d+1} \in A_j') = \frac{\text{Vol}(A_j')}{r^d \chi \log n} \leq \left[ \frac{m \nu_d}{r^d \chi} \Bigg( \bigg(\overline{t} + \frac{\gamma(\epsilon)}{2} + \frac{\sqrt{d}}{2}s^{1/d}\bigg)^d - \bigg(\overline{t} - \frac{\gamma(\epsilon)}{2} - \frac{\sqrt{d}}{2}s^{1/d}\bigg)^d \Bigg) \right]^{d+1}.
    \end{equation*}
    Since $\gamma(\epsilon) \to 0$ as $\epsilon \to 0$ by assumption, we can choose $\epsilon > 0$ such that $\gamma(\epsilon) \leq \sqrt{d}s^{1/d}$, which yields
    \begin{equation}
    \label{eq:probability_in_Aj'}
        \P(u_1, u_2, \ldots, u_{d+1} \in A_j') \leq \left[ \frac{m \nu_d}{r^d \chi} \bigg( (\overline{t} + \sqrt{d}s^{1/d})^d - (\overline{t} - \sqrt{d}s^{1/d})^d \bigg) \right]^{d+1}.
    \end{equation}
    
    Now, we show how (\ref{eq:probability_in_Aj'}) gives us the desired upper bound for $\P(R(S) \neq \varnothing)$. In the preceding discussion, we established that $\P(R(S) \cap C_j \neq \varnothing) \leq \P(u_1, u_2, \ldots, u_{d+1} \in A_j')$. Therefore, we obtain
    \begin{equation*}
        \P(R(S) \cap C_j \neq \varnothing) \leq \left[ \frac{m \nu_d}{r^d \chi} \bigg( (\overline{t} + \sqrt{d}s^{1/d})^d - (\overline{t} - \sqrt{d}s^{1/d})^d \bigg) \right]^{d+1}.
    \end{equation*}
    Then, since we divided $U_i$ into $(2r)^d/s$ hypercubes $C_j$, we can apply the union bound to obtain
    \begin{equation*}
    \begin{aligned}
        \P(R(S) \neq \varnothing) &\leq \frac{(2r)^d}{s} \P(R(S) \cap C_j \neq \varnothing) \\
        &\leq \frac{(2r)^d}{s}\left[ \frac{m \nu_d}{r^d \chi} \bigg( (\overline{t} + \sqrt{d}s^{1/d})^d - (\overline{t} - \sqrt{d}s^{1/d})^d \bigg) \right]^{d+1}.
    \end{aligned}
    \end{equation*}
    Letting $s' = \sqrt{d}s^{1/d}$, we obtain 
    \begin{equation*}
    \begin{aligned}
        \P(R(S) \neq \varnothing) &\leq \frac{(2r\sqrt{d})^d}{(s')^d} \left[ \frac{m \nu_d}{r^d \chi} \bigg( (\overline{t} + s')^d - (\overline{t} - s')^d \bigg) \right]^{d+1} \\
        &= \frac{(2r\sqrt{d})^d(m \nu_d)^{d+1}}{(r^d \chi)^{d+1}} \cdot \frac{1}{(s')^d} \bigg[ (\overline{t}+s')^d - (\overline{t}-s')^d \bigg]^{d+1}
    \end{aligned}
    \end{equation*}
    Now, observe that the first term is constant and the second term goes to zero as $s' \to 0$ due to Lemma \ref{lem:limit_calculation_lemma}. Hence, we can conclude that
    \begin{equation*}
        \lim_{s' \to 0} \P(R(S) \neq \varnothing) = 0
    \end{equation*}
    Therefore, we can choose $s > 0$ such that $\P(R(S) \neq \varnothing) < \kappa$, which implies that $\P(\calE(S)) \geq 1 - \kappa$.

    Now, we show that block $B_i$ is $(\delta/(2(d+1)), \epsilon)$-distinguishing with probability $1 - o(1/n)$. Since $B_i$ is $\delta$-occupied, we can partition $V_i$ into at least $\floor{\delta \log n / (d+1)}$ sets of $d+1$ vertices, which we call $S_1, S_2, \ldots, S_{\floor{\delta \log n / (d+1)}}$. For clarity, we will omit the floor function for the remainder of the proof. Now, we define $N$ as the number of sets $S_j$ for which $\calE(S_j)$ holds. Observe that $N$ lower-bounds the number of vertices that $\epsilon$-distinguish any given point $v \in U_i$ because if $\calE(S_j)$ holds, then there's at least one vertex in $S_j$ which $\epsilon$-distinguishes $v$. Therefore, we just need to show that $N \geq \delta \log n / (2(d+1))$ with probability $1 - o(1/n)$. 
    
    Recall that for any $\kappa > 0$, we have that $\P(\calE(S_j)) > 1 - \kappa$ for any given set $S_j$. In addition, note that $\calE(S_j)$ and $\calE(S_k)$ are independent for $S_j \neq S_k$ because the locations of the vertices in $B_i$ are independent. Therefore, we have that $N \succeq \text{Binom}(\delta \log n/(d+1), 1 - \kappa)$. Then, we can apply the Chernoff bound (Lemma \ref{lem:binomial_Chernoff}) to obtain
    \begin{align}
    \label{eq:lack_distinguishing_vertices_prob}
        \P\left(\text{Binom}\left(\frac{\delta \log n}{d+1}, 1 - \kappa\right) \leq \frac{\delta \log n}{2(d+1)}\right) &= \P\left(\text{Binom}\left(\frac{\delta \log n}{d+1}, \kappa \right) \geq \frac{\delta \log n}{2(d+1)}\right) \notag \\
        &\leq \left( \frac{e^{(1/(2\kappa)) - 1}}{(1/(2\kappa))^{(1/(2\kappa))}}\right)^{(\delta \kappa \log n)/(d+1)} \notag \\
        &= \exp\left(\log\left(\frac{e^{(1/(2\kappa)) - 1}}{(1/(2\kappa))^{(1/(2\kappa))}}\right) \frac{\delta \kappa \log n}{d+1}\right) \notag \\
        &= \exp\left(\left(\frac{1}{2\kappa} - 1 - \frac{1}{2\kappa}\log\left(\frac{1}{2\kappa}\right)\right) \frac{\delta \kappa \log n}{d+1}\right) \notag \\
        &= \exp\left( \left(1 - 2\kappa - \log\left(\frac{1}{2\kappa}\right) \right) \frac{\delta}{2(d+1)} \right) \notag \\
        &= n^{\big( 1 - 2\kappa - \log(1/(2\kappa)) \big) \delta/(2(d+1))}.
    \end{align}
    Choosing $\kappa > 0$ small enough, we obtain that \eqref{eq:lack_distinguishing_vertices_prob} is $o(1/n)$. Then, recalling that $N \succeq \text{Binom}(\delta \log n/(d+1), 1 - \kappa)$, we have
    \begin{equation*}
        \P\left(N \leq \frac{\delta \log n}{2(d+1)}\right) = o(1/n).
    \end{equation*}
    
    Therefore, we see that $N \geq (\delta \log n) / (2(d+1))$ with probability $1 - o(1/n)$, which proves that $B_i$ is $(\delta/(2(d+1)), \epsilon)$-distinguishing with probability $1 - o(1/n)$. Finally, since there are $n / (r^d \chi \log n)$ total blocks, we can apply the union bound to obtain that all $\delta$-occupied blocks are $(\delta/2(d+1)), \epsilon)$-distinguishing with probability $1 - o(1)$.
\end{proof} 

Now, we show that $\hat{\sigma}$ makes a constant number of mistakes in all $\delta$-occupied blocks with probability $1 - o(1)$. To show this, we will adapt the notations and proofs in Section 4.3 of \cite{Gaudio+2024}. Let $C_{p(i)}$ denote the location, true label $\sigma^*(v)$, and estimated label $\hat{\sigma}(v)$ of each vertex in $v \in V_{p(i)}$, which we call the configuration of $V_{p(i)}$.

\begin{prop}
\label{prop:propagation_error_vertex}
Let $\delta > 0$ and $\epsilon > 0$ be constants. Define $M := 65(d+1)/(\epsilon^2 \delta)$, $\eta_2 := \exp(M \epsilon (1+\epsilon)/4 - M^2 (1+\epsilon)^2/8)$, and $c_2 := \epsilon^2\delta / (64(d+1))$. Let $B_i$ be a $\delta$-occupied block with vertices $V_i$, and suppose that its parent block is $B_{p(i)}$ with vertices $V_{p(i)}$. Suppose that $V_{p(i)}$ has configuration $C_{p(i)}$ such that $\hat{\sigma}$ makes at most $M$ mistakes on $V_{p(i)}$ and there are $n_2 \geq \delta \log n / (2(d+1))$ vertices in $V_{p(i)}$ which $\epsilon$-distinguish any given $v \in V_i$. Suppose that $\hat{\sigma}$ labels $V_i$ using \texttt{Propagate} on input $(G, V_{p(i)}, V_i, f_\text{in}, f_\text{out})$. Then, the probability that $\hat{\sigma}$ labels any given vertex $v \in V_i$ incorrectly is
\begin{equation*}
    \P(\hat{\sigma}(v) \neq \sigma^*(u_0) \sigma^*(v) \mid C_{p(i)} ) \leq \eta_2 n^{-c_2}.
\end{equation*}
Furthermore, suppose that $|V_i| \leq \Delta \log n$ for some constant $\Delta$, and let $\eta_3 := e^M (\eta_2 \Delta / M)^M$. Then, the probability that $\hat{\sigma}$ makes more than $M$ mistakes on $V_i$ is 
\begin{equation*}
    \P\bigg( | \{v \in V_i: \hat{\sigma}(v) \neq \sigma^*(u_0) \sigma^*(v)\} | > M \: \Big| \: C_{p(i)}, |V_i| \leq \Delta \log n \bigg) \leq \eta_3 n^{-129/128}
\end{equation*}
\end{prop}
\begin{proof}
First, consider the case where $\hat{\sigma}$ labels more vertices that $\epsilon$-distinguish $v$ as $+1$ than $-1$ in $V_{p(i)}$. We start by using the law of total probability, which shows that
\begin{equation}
\label{eq:propagate_mistake_prob}
    \P(\hat{\sigma}(v) \neq \sigma^*(u_0)\sigma^*(v) \mid C_{p(i)} ) = \frac{1}{2}\P(Y_v \leq 0 \mid C_{p(i)}, \sigma^*(v) = \sigma^*(u_0)) + \frac{1}{2}\P(Y_v \geq 0 \mid C_{p(i)}, \sigma^*(v) \neq \sigma^*(u_0))
\end{equation}
Now, we analyze each term individually. Let us consider the first term $\P(Y_v \leq 0 \mid C_{p(i)}, \sigma^*(v) = \sigma^*(u_0))$. To bound this probability, we first compute the expectation $\E[Y_v \mid C_{p(i)}, \sigma^*(v) = \sigma^*(u_0)]$. Define the sets 
\begin{equation*}
    R_+(v) := \{u \in V_{p(i)} : \hat{\sigma}(u) = +1, u \dist v, \sigma^*(u) = \sigma^*(u_0) \}
\end{equation*}
and
\begin{equation*}
    R_-(v) := \{u \in V_{p(i)}, \hat{\sigma}(u) = +1, u \dist v, \sigma^*(u) \neq \sigma^*(u_0)\}.
\end{equation*} 
That is, out of all the vertices $u$ in the parent block $V_{p(i)}$ which $\hat{\sigma}$ labels as $+1$ and which $\epsilon$-distinguish $v$, the set $R_+(v)$ contains those which have the same ground truth label as $u_0$ while the set $R_-(v)$ contains those which have the opposite ground truth label as $u_0$. 
Letting $L_{p(i)}$ and $L_i$ denote the locations of the vertices in $V_{p(i)}$ and $V_i$ respectively, observe that
\begin{align}
\label{eq:condition_expectation_Yv}
    & \E[ Y_v \mid C_{p(i)}, \sigma^*(v) = \sigma^*(u_0) ] \notag \\
    &= \sum_{u \in R_+(v)} \beta_u \Big( \P(u \sim v \mid L_{p(i)}, L_i, \sigma^*(u) = \sigma^*(u_0), \sigma^*(v) = \sigma^*(u_0)) - \P(u \sim v \mid L_{p(i)}, L_i) \Big) \notag \\
    &\qquad + \sum_{u \in R_-(v)} \beta_u \Big( \P(u \sim v \mid L_{p(i)}, L_i, \sigma^*(u) \neq \sigma^*(u_0), \sigma^*(v) = \sigma^*(u_0)) - \P(u \sim v \mid L_{p(i)}, L_i) \Big) \notag \\
    &= \sum_{u \in R_+(v)} \frac{1}{2}\beta_u \Big(\overline{f}_\text{in}(\|u-v\|) - \overline{f}_\text{out}(\|u-v\|) \Big) + \sum_{u \in R_-(v)} \frac{1}{2}\beta_u \Big(\overline{f}_\text{out}(\|u-v\|) - \overline{f}_\text{in}(\|u-v\|) \Big) \notag \\
    &= \sum_{u \in R_+(v)} \frac{1}{2} \Big| \overline{f}_\text{in}(\|u-v\|) - \overline{f}_\text{out}(\|u-v\|) \Big| - \sum_{u \in R_-(v)} \frac{1}{2} \Big| \overline{f}_\text{in}(\|u-v\|) - \overline{f}_\text{out}(\|u-v\|) \Big|.
\end{align}
We examine each summation of (\ref{eq:condition_expectation_Yv}) separately. For the first summation, we have that $|\overline{f}_\text{in}(\|u-v\|) - \overline{f}_\text{out}(\|u-v\|)| > \epsilon$ for each term in the summation because we are summing over $u \dist v$. Since we are considering the case where $\hat{\sigma}$ labels at least half of the vertices which $\epsilon$-distinguish $v$ as $+1$, there are at least $n_2/2$ vertices in $V_{p(i)}$ which $\hat{\sigma}$ labels as $+1$ and $\epsilon$-distinguish $v$. At most $M$ of these are mistakes, so we obtain
\begin{equation*}
    \sum_{u \in R_+(v)} \frac{1}{2} \Big| \overline{f}_\text{in}(\|u-v\|) - \overline{f}_\text{out}(\|u-v\|) \Big| > \frac{1}{2}\left(\frac{n_2}{2} - M\right) \epsilon.
\end{equation*}
For the second summation, we use the fact that there are at most $M$ mistakes, which yields the trivial bound
\begin{equation*}
    \sum_{u \in R_-(v)} \frac{1}{2} \Big| \overline{f}_\text{in}(\|u-v\|) - \overline{f}_\text{out}(\|u-v\|) \Big| < \frac{1}{2}M.
\end{equation*}
Therefore, we obtain that
\begin{equation}
\label{eq:condition_expectation_Yv_simplified}
\begin{aligned}
    \E[Y_v \mid C_{p(i)}, \sigma^*(v) = \sigma^*(u_0)] &> \frac{1}{2}\left(\frac{n_2}{2} - M\right) \epsilon - \frac{1}{2}M \\
    &= \frac{1}{4} n_2 \epsilon -\frac{1}{2}M(1 + \epsilon).
\end{aligned}
\end{equation}
Now, we can apply Hoeffding's inequality. Since $Y_v$ is the sum of independent random variables bounded between $-1$ and $1$, we have that
\begin{equation*}
\begin{aligned}
    \P(Y_v \leq 0 \mid C_{p(i)}, \sigma^*(v) = \sigma^*(u_0)) &\leq \exp \left( -2 \frac{(n_2 \epsilon / 4 - M(1+\epsilon)/2)^2}{4n_2} \right) \\
    &= \exp\left(- \frac{(n_2\epsilon - 2M(1+\epsilon))^2}{32n_2}\right) \\
    &= \exp\left(- \frac{\epsilon^2}{32}n_2 + \frac{M \epsilon (1+\epsilon)}{8} - \frac{M^2(1+\epsilon)^2}{8n_2}\right) \\
    &\leq \exp\left(- \frac{\epsilon^2}{32}n_2 + \frac{M \epsilon (1+\epsilon)}{8} - \frac{M^2(1+\epsilon)^2}{8}\right) \\
    &= \exp\left(\frac{M \epsilon (1+\epsilon)}{8} - \frac{M^2(1+\epsilon)^2}{8}\right) \exp\left(- \frac{\epsilon^2}{32} n_2\right) \\
    &= \eta_2 \exp\left(-\frac{\epsilon^2}{32}n_2\right)
\end{aligned}
\end{equation*}
where the second inequality follows from $n_2 \geq 1$ because it represents a number of vertices. Then, we use the fact that $n_2 \geq \delta \log n / (2(d+1))$ when conditioned on $C_{p(i)}$, which yields
\begin{align}
\label{eq:propagate_error_prob_cond1}
    \P(Y_v \leq 0 \mid C_{p(i)}, \sigma^*(v) = \sigma^*(u_0)) &\leq \eta_2 \exp\left(-\frac{\epsilon^2 \delta}{64(d+1)} \log n\right) \notag \\
    &= \eta_2 n^{-\epsilon^2\delta/(64(d+1))} \notag \\
    &= \eta_2 n^{-c_2}.
\end{align}
Therefore, we obtain an upper bound for the first term in (\ref{eq:propagate_mistake_prob}). Using a similar argument, we can also show that 
\begin{equation}
\label{eq:propagate_error_prob_cond2}
    \P(Y_v \geq 0 \mid C_{p(i)}, \sigma^*(v) \neq \sigma^*(u_0)) \leq \eta_2 n^{-c_2},
\end{equation}
which provides an upper bound for the second term in (\ref{eq:propagate_mistake_prob}). Therefore, we can substitute (\ref{eq:propagate_error_prob_cond1}) and (\ref{eq:propagate_error_prob_cond2}) into (\ref{eq:propagate_mistake_prob}) to obtain that
\begin{equation*}
    \P(\hat{\sigma}(v) \neq \sigma^*(u_0)\sigma^*(v) \mid C_{p(i)}) \leq \eta_2 n^{-c_2}.
\end{equation*}
Now, we consider the other case where $\hat{\sigma}$ labels more vertices as $-1$ than $+1$ in $V_{p(i)}$; the analysis is virtually identical. Define the sets 
\begin{equation*}
    S_+(v) := \{u \in V_{p(i)} : \hat{\sigma}(u) = -1, u \dist v, \sigma^*(u) = \sigma^*(u_0) \}
\end{equation*}
and
\begin{equation*}
    S_-(v) := \{u \in V_{p(i)}, \hat{\sigma}(u) = -1, u \dist v, \sigma^*(u) \neq \sigma^*(u_0)\}.
\end{equation*}
Using a similar calculation as the one leading up to \eqref{eq:condition_expectation_Yv_simplified}, we obtain
\begin{equation*}
\begin{aligned}
    & \E[ Y_v \mid C_{p(i)}, \sigma^*(v) = \sigma^*(u_0) ] \\
    &= \sum_{u \in S_+(v)} \frac{1}{2} \Big| \overline{f}_\text{in}(\|u-v\|) - \overline{f}_\text{out}(\|u-v\|) \Big| - \sum_{u \in S_-(v)} \frac{1}{2} \Big| \overline{f}_\text{in}(\|u-v\|) - \overline{f}_\text{out}(\|u-v\|) \Big| \\
    &> \frac{1}{4}n_2 \epsilon - \frac{1}{2}M(1+\epsilon).
\end{aligned}
\end{equation*}
Then, we can apply Hoeffding's inequality. Using the same calculation as the one leading up to \eqref{eq:propagate_error_prob_cond1} yields
\begin{equation*}
    \P(Y_v \leq 0 \mid C_{p(i)}, \sigma^*(v) = \sigma^*(u_0)) \leq \eta_2 n^{-c_2}.
\end{equation*}
A similar argument shows that 
\begin{equation*}
    \P(Y_v \geq 0 \mid C_{p(i)}, \sigma^*(v) \neq \sigma^*(u_0)) \leq \eta_2 n^{-c_2}.
\end{equation*}
Therefore,
\begin{equation*}
    \P(\hat{\sigma}(v) \neq \sigma^*(u_0)\sigma^*(v) \mid C_{p(i)}) \leq \eta_2 n^{-c_2}.
\end{equation*}

This provides an upper-bound on the probability of labeling one vertex incorrectly. Now, we upper-bound the probability of making more than $M$ mistakes on $V_i$. Let $K_i := | \{v \in V_i: \hat{\sigma}(v) \neq \sigma^*(u_0) \sigma^*(v)\} |$ be the number of mistakes that $\hat{\sigma}$ makes on $V_i$, and let $\calE_v := \{\hat{\sigma}(v) \neq \sigma^*(u_0) \sigma^*(v)\}$ be the event that a particular vertex $v$ is mislabeled by $\hat{\sigma}$. Observe that conditioned on $C_{p(i)}$, the events $\calE_v$ for $v \in V_i$ are independent because we use disjoint sets of edges, which are independently generated conditioned on $C_{p(i)}$, to classify different vertices. Therefore, conditioned on $C_{p(i)}$ and $|V_i| \leq \Delta \log n$, we have that $K_i \preceq \text{Binom}(\Delta \log n, \eta_2 n^{-c_2})$. Letting $\mu := \eta_2 n^{-c_2} \Delta \log n$, we then apply the Chernoff bound (Lemma \ref{lem:binomial_Chernoff}) to obtain
\begin{equation*}
\begin{aligned}
    \P(K_i > M \mid C_{p(i)}, |V_i| \leq \Delta \log n) &\leq \bigg(\frac{e^{(M/\mu) - 1}}{(M/\mu)^{(M/\mu)}}\bigg)^\mu \\
    &\leq e^M \Big(\frac{\mu}{M}\Big)^M \\
    &= e^M \Big(\frac{\eta_2 n^{-c_2} \Delta \log n}{M}\Big)^M \\
    &= e^M \Big(\frac{\eta_2 \Delta}{M}\Big)^M (\log n)^M n^{-c_2 M}.
\end{aligned}
\end{equation*}
Examining each term in the product above, we see that $e^M$ and $(\eta_2 \Delta / M)^M$ are constants in $n$. Therefore, we obtain that
\begin{equation*}
    \P(K_i > M \mid C_{p(i)}, |V_i| \leq \Delta \log n) \leq \eta_3 (\log n)^M n^{-c_2 M}
\end{equation*}
where $\eta_3 := e^M (\eta_2 \Delta / M)^M$. Then, we make two observations: First, observe that $n^{-c_2 M} = n^{-65/64}$ because $c_2 = (\epsilon^2 \delta) / (64(d+1))$ while $M = (65(d+1))/(\epsilon^2 \delta)$. Second, we know that $(\log n)^M < n^{1/128}$ for sufficiently large $n$. Therefore, we have that
\begin{equation*}
    \P(K_i > M \mid C_{p(i)}, |V_i| \leq \Delta \log n) \leq \eta_3 n^{-129/128} = o(1/n).
\end{equation*}
which shows that the probability of making more than $M$ mistakes on $V_i$ is $o(1/n)$.
\end{proof}

Using Proposition \ref{prop:propagation_error_vertex}, we can show that the Phase I labeling $\hat{\sigma}$ makes at most $M$ mistakes in all $\delta$-occupied blocks with probability $1 - o(1)$.
To apply Proposition \ref{prop:propagation_error_vertex}, we need all blocks to have at most $\Delta \log n$ vertices for some constant $\Delta$. Fortunately, this is true with high probability. Using a straightforward modification of Lemma 4.13 in \cite{Gaudio+2024}, we obtain the following result.
\begin{lemma}
\label{lem:maximum_vertices_per_block}
    For the blocks obtained in Line \ref{line:create_blocks} in Algorithm \ref{alg:exact_recovery}, there exists an explicitly computable constant $\Delta > 0$ such that 
    \begin{equation*}
        \P\bigg( \bigcap_{i = 1}^{n / (r^d \chi \log n)} \Big\{ |V_i| \leq \Delta \log n \Big\} \bigg) = 1 - o(1).
    \end{equation*}
\end{lemma}

Now, we prove that the Phase I labeling $\hat{\sigma}$ makes at most $M$ mistakes in all $\delta$-occupied blocks with high probability.

\begin{theorem}
\label{thm:propagation_error_total}
Let $G \sim \text{GSBM}(\lambda, n, r, f_\text{in}, f_\text{out}, d)$ such that $\lambda r > 1$ if $d = 1$ and $\lambda \nu_d r^d > 1$ if $d \geq 2$. Fix $\eta > 0$, and define $K := \nu_d(1 + \sqrt{d}\chi^{1/d})/\chi$. Suppose that $\chi$, $\delta$ satisfy \eqref{eq:chi_delta_conditions_d=1} if $d = 1$ or \eqref{eq:chi_delta_conditions_d>=2} if $d \geq 2$, and also that $\delta < \eta / K$. Let $\hat{\sigma}$ be the labeling obtained from Phase I of Algorithm \ref{alg:exact_recovery}. Then, there exists a constant $M$ such that
\begin{equation}
\label{eq:errors_within_block}
    \P\bigg( \bigcap_{i \in V^\dagger} \Big\{ |v \in V_i: \hat{\sigma}(v) \neq \sigma^*(u_0) \sigma^*(v) | \leq M \Big\} \bigg) = 1 - o(1)
\end{equation}
which means that $\hat{\sigma}$ makes at most $M$ mistakes on all $\delta$-occupied blocks with probability $1 - o(1)$. Consequently, $\hat{\sigma}$ achieves almost-exact recovery. Furthermore,
\begin{equation}
\label{eq:errors_within_neighborhood}
    \P\bigg( \bigcap_{v \in V} \{|u \in \calN(v): \hat{\sigma}(u) \neq \sigma^*(u_0)\sigma^*(u)| \leq \eta \log n\} \bigg) = 1 - o(1).
\end{equation}
meaning that $\hat{\sigma}$ makes at most $\eta \log n$ mistakes in the neighborhood of each vertex with probability $1 - o(1)$.
\end{theorem}
\begin{proof}
Let $\epsilon$ be such that Lemma \ref{lem:distinguishing_block_lemma} is satisfied and $\Delta > 0$ be such that Lemma \ref{lem:maximum_vertices_per_block} is satisfied. We define the events 
\begin{equation*}
\begin{aligned}
    &\calI = \text{\{All blocks have at most $\Delta \log n$ vertices\},} \\
    &\calJ = \text{\{Every $\delta$-occupied block is $(\delta \log n /(2(d+1)), \epsilon)$-distinguishing\}, and} \\
    &\calH = \text{\{The $(r^d \chi \log n, \delta \log n)$-visibility graph of $G$ is connected\}}. \\
\end{aligned}
\end{equation*}
Furthermore, let $\calA_i$ be the event that $\hat{\sigma}$ makes at most $M := 65(d+1)/(\epsilon^2 \delta)$ mistakes in block $B_i$, meaning that
\begin{equation*}
    \calA_i = \Big\{ |v \in V_i: \hat{\sigma}(v) \neq \sigma^*(u_0) \sigma^*(v) | \leq M \Big\}. 
\end{equation*}
Recalling that $i_1, i_2, \ldots$ is the visitation order during propagation, observe that
\begin{equation}
\label{eq:Ai_all_blocks}
\begin{aligned}
    \P\bigg( \bigcap_{i \in V^\dagger} \calA_i \bigg) &\geq \P\bigg( \bigcap_{i \in V^\dagger} \calA_i \mid \calI \cap \calJ \cap \calH \bigg) \P(\calI \cap \calJ \cap \calH) \\
    &= \P(\calA_{i_1} \mid \calI \cap \calJ \cap \calH) \bigg( \prod_{j=2}^{|V^\dagger|} \P(\calA_{i_j} \mid \calA_{i_1}, \ldots, \calA_{i_{j-1}}, \calI \cap \calJ \cap \calH) \bigg) \P(\calI \cap \calJ \cap \calH).
\end{aligned}
\end{equation}
For the first factor, we can use Proposition \ref{prop:initial_block_labeling} to obtain
\begin{equation}
\label{eq:Ai_first_block}
    \P(\calA_{i_1} \mid \calI \cap \calJ \cap \calH) \geq 1 - \eta_1n^{-c_1} \Delta \log n.
\end{equation}
for constants $\eta_1, c_1 > 0$. For the second factor, we condition on the configuration $C_{p(i_j)}$ of the vertices $V_{p(i_j)}$ and use the law of total expectation. Define the event $\calX_j := \{\calA_{i_1} \cap \ldots \cap \calA_{i_{j-1}} \cap \calI \cap \calJ \cap \calH\}$. Then, we obtain that
\begin{equation*}
\begin{aligned}
    \P(\calA_{i_j} \mid \calX_j) &= \E[\1(\calA_{i_j}) \mid \calX_j] \\
    &= \E[ \E[\1(\calA_{i_j}) \mid \calX_j, C_{p(i_j)} ] \mid \calX_j ] \\
    &= \E[ \P(\calA_{i_j} \mid \calX_j, C_{p(i_j)}) \mid \calX_j]
\end{aligned}
\end{equation*}
Now, observe that conditioned on the configuration $C_{p(i_j)}$, the event $\calA_{i_j}$ is independent of $\calX_j$ because the vertices in $V_{i_j}$ are labeled with respect to the vertices in $V_{p(i_j)}$. Therefore, we have that
\begin{equation*}
    \P(\calA_{i_j} \mid \calX_j) = \E[ \P(\calA_{i_j} \mid C_{p(i_j)}) \mid \calX_j ]
\end{equation*}
Then, observe that conditioned on $\calX_j$, the configuration $C_{p(i_j)}$ satisfies the following properties: $\hat{\sigma}$ makes at most $M$ mistakes on $V_{p(i_j)}$ due to $\calA_{p(i_j)}$, there are at least $\delta \log n / (2(d+1))$ vertices in $V_{p(i_j)}$ which $\epsilon$-distinguish any given vertex  in $V_{i_j}$ due to $\calJ$, and there are at most $\Delta \log n$ vertices in $V_{i_j}$ due to $\calI$. Thus, we can apply Proposition \ref{prop:propagation_error_vertex} to obtain that for any configuration $C_{p(i_j)}$ satisfying $\calX_j$, we have
\begin{equation*}
    \P(\calA_{i_j} \mid C_{p(i_j)}) \geq 1 - \eta_3 n^{-129/128}
\end{equation*}
for a constant $\eta_3 > 0$. Therefore, we obtain that
\begin{equation}
\label{eq:Ai_remaining_blocks}
    \P(\calA_{i_j} \mid \calX_j) \geq 1 - \eta_3n^{-129/128}.
\end{equation}

For the third factor, we combine Proposition \ref{prop:connected_prop}, Lemma \ref{lem:distinguishing_block_lemma}, and Lemma \ref{lem:maximum_vertices_per_block} to obtain
\begin{equation}
\label{eq:IJK_prob}
    \P(\calI \cap \calJ \cap \calH) = 1 - o(1).
\end{equation}
Therefore, we can substitute \eqref{eq:Ai_first_block}, \eqref{eq:Ai_remaining_blocks}, and \eqref{eq:IJK_prob} into \eqref{eq:Ai_all_blocks}, and use Bernoulli's inequality to obtain
\begin{equation*}
\begin{aligned}
    \P\bigg( \bigcap_{i \in V^\dagger} \calA_i \bigg) &\geq (1 - \eta_1 n^{-c_1} \Delta \log n)\left(1 - \eta_3 n^{-129/128}\right)^{n / (r^d \chi \log n)} (1 - o(1))\\
    &\geq (1 - \eta_1 n^{-c_1} \Delta \log n)\left(1 - \frac{\eta_3 n^{-1/128}}{r^d \chi \log n}\right) (1 - o(1)) \\
    &= 1 - o(1),
\end{aligned}
\end{equation*}
which shows (\ref{eq:errors_within_block}). Since $\delta$ can be arbitrarily small, we obtain that $\hat{\sigma}$ achieves almost-exact recovery.

Now, we show \eqref{eq:errors_within_neighborhood}. From (\ref{eq:errors_within_block}), we use the fact that $\delta \log n > M$ for sufficiently large $n$ to obtain 
\begin{equation*}
    \P\bigg( \bigcap_{i = 1}^{n/(r^d \chi \log n)} \Big\{ |v \in V_i: \hat{\sigma}(v) \neq \sigma^*(u_0) \sigma^*(v) | \leq \delta \log n \Big\} \bigg) = 1 - o(1)
\end{equation*}
because any unoccupied block automatically has less than $\delta \log n$ mistakes. Then, we note that $K$ is an upper bound for the number of blocks which intersect the neighborhood $\calN(v)$ for any given vertex $v$. Therefore, we have that
\begin{equation*}
    \P\bigg( \bigcap_{v \in V} \{|u \in \calN(v): \hat{\sigma}(u) \neq \sigma^*(u_0)\sigma^*(u)| \leq \delta K \log n\} \bigg) = 1 - o(1).
\end{equation*}
Since $\delta < \eta / K$, we obtain (\ref{eq:errors_within_neighborhood}).
\end{proof}

\section{Refine}
\label{sec:refine}
We will now show that the Phase II estimator $\tilde{\sigma}$ produced by Algorithm \ref{alg:exact_recovery} achieves exact recovery. In particular, we show the probability that $\tau(v, \hat{\sigma})$ as defined in \eqref{eq:def_tau} makes a mistake on any given vertex $v$ is $o(1/n)$, which implies that $\tilde{\sigma}$ makes a mistake with probability $o(1)$. To accomplish this, we first upper-bound for arbitrary fixed $\rho, \eta > 0$ the probability that $\tau(v, \sigma^*)$ comes within $\rho \eta \log n$ of making an error is $n^{-(I(f_\text{in}, f_\text{out}) - \rho \eta /2)}$. Then, we will show that for any $\sigma$ which differs from $\sigma^*$ by at most $\eta \log n$ vertices in the neighborhood of $v$, the difference $|\tau(v, \sigma^*(u_0)\sigma) - \tau(v, \sigma^*)| \leq \rho \eta \log n$ for an appropriate constant $\rho$. Therefore, the probability $\tau(v, \sigma^*(u_0) \sigma)$ makes an error on $v$ for any such $\sigma$ is at most $n^{-(I(f_\text{in}, f_\text{out}) - \rho \eta /2)}$ as well. Finally, the condition $I(f_\text{in}, f_\text{out}) > 1$ means that we can take $\eta$ sufficiently small such that the error probability is $o(1/n)$.

To start, we define a random variable $Z$ such that
\begin{equation*}
    (\tau(v, \sigma^*) \mid \sigma^*(v) = -1) \sim Z
\end{equation*}
and 
\begin{equation*}
    (\tau(v, \sigma^*) \mid \sigma^*(v) = +1) \sim -Z.
\end{equation*}
We will express $Z$ as a random sum of i.i.d. random variables. To this end, let $D$ be a random variable with density
\begin{equation}
\label{eq:def_D}
    f_D(x) :=
    \begin{cases}
        dx^{d-1}/r^d \text{ if } 0 \leq x \leq r \\
        0 \text{ otherwise.} 
    \end{cases}
\end{equation}
To interpret $D$, we see that if a vertex $u$ is dropped uniformly at random in a $d$-dimensional sphere of radius $r (\log n)^{1/d}$ centered at $v$, then $D$ represents the distance $\|u - v \| / (\log n)^{1/d}$. Let $U_\text{out}$ be a random variable such that $(U_\text{out} \mid D) \sim \text{Bern}(f_\text{out}(D))$ and $U_\text{in}$ be such that $(U_\text{in} \mid D) \sim \text{Bern}(f_\text{in}(D))$. Let $N \sim \text{Pois}(\lambda \nu_d r^d \log n /2)$ and let $N_+, N_- \sim N$ be independent. Then, we define
\begin{equation}
\label{eq:def_P}
    P := \log\left(\frac{f_\text{in}(D)}{f_\text{out}(D)}\right)U_\text{out} + \log\left(\frac{1 - f_\text{in}(D)}{1 - f_\text{out}(D)}\right)(1-U_\text{out})
\end{equation}
and
\begin{equation}
\label{eq:def_Q}
    Q := \log\left(\frac{f_\text{out}(D)}{f_\text{in}(D)}\right)U_\text{in} + \log\left(\frac{1 - f_\text{out}(D)}{1 - f_\text{in}(D)}\right)(1-U_\text{in}).
\end{equation}
Finally, let $\{P_i\}_{i \in \N} \sim P$ be i.i.d. and $\{Q_i\}_{i \in \N} \sim Q$ be i.i.d, and define
\begin{equation}
\label{eq:def_Z}
    Z = \sum_{i=1}^{N_+} P_i + \sum_{j=1}^{N_-} Q_j
\end{equation}
where $N_+$, $N_-$, $P_i$, and $Q_i$ are all independent. Now, we can upper-bound the probability of $\tau(v, \sigma^*)$ being within $\rho \eta \log n$ of making an error.

\begin{lemma}
\label{lem:refine_Chernoff}
    For any constants $\rho > 0 $ and $\eta > 0$, 
    \begin{equation*}
        \P(Z \geq -\rho \eta \log n) \leq n^{-( I(f_\text{in}, f_\text{out}) - \rho\eta/2 )}.
    \end{equation*}
    Consequently, 
    \begin{equation}
    \label{eq:tau-_incorrect_prob}
        \P(\tau(v, \sigma^*) \geq -\rho \eta \log n \mid \sigma^*(v)= -1) \leq n^{-( I(f_\text{in}, f_\text{out}) - \rho\eta/2 )}
    \end{equation}
    and
    \begin{equation}
    \label{eq:tau+_incorrect_prob}
        \P(\tau(v, \sigma^*) \leq \rho \eta \log n \mid \sigma^*(v)= +1) \leq n^{-( I(f_\text{in}, f_\text{out}) - \rho\eta/2 )}.
    \end{equation}
\end{lemma}
\begin{proof}
We start by deriving a Chernoff bound for $Z$, which yields
\begin{equation}
\label{eq:Chernoff_Z}
    \P(Z \geq -\rho \eta \log n) = \P(e^{tZ} \geq n^{- t \rho \eta}) \leq n^{t \rho \eta} \E[e^{tZ}]
\end{equation}
for all $t > 0$. Now, we compute the moment-generating function (MGF)  of $Z$. Let $M_N$, $M_P$, and $M_Q$ be the MGFs of $N$, $P$, and $Q$ respectively. Observe that
\begin{equation*}
\begin{aligned}
    \E[e^{tZ}] &= \E\bigg[\exp\bigg(t \bigg(\sum_{i=1}^{N_+} P_i + \sum_{j=1}^{N-} Q_j \bigg)\bigg)\bigg] \\
    &= \E\bigg[\exp\bigg(t\sum_{i=1}^{N_+} P_i\bigg) \exp\bigg(t\sum_{j=1}^{N-} Q_j\bigg)\bigg] \\
    &= \E\bigg[\exp\bigg(t\sum_{i=1}^{N+} P_i\bigg)\bigg] \cdot \E\bigg[\exp\bigg(t\sum_{j=1}^{N-} Q_j\bigg)\bigg] \\
    &= M_N(\log M_P(t)) \cdot M_N(\log M_Q(t))
\end{aligned}
\end{equation*}
where the second-to-last line follows from the independence of $N_+$, $N_-$, $P_i$, and $Q_i$, and the last line follows from the definition and properties of the MGF. Since $N \sim \text{Pois}(\lambda \nu_d r^d \log n / 2)$, we have an explicit expression for its MGF, which is $M_N(t) = \exp( \lambda \nu_d r^d \log n (e^t-1) / 2 )$. Therefore, we obtain that
\begin{equation}
\label{eq:MGF_Z}
\begin{aligned}
    \E[e^{tZ}] &= \exp\bigg( \frac{\lambda \nu_d r^d \log n}{2} (M_P(t)-1) \bigg) \cdot \exp\bigg( \frac{\lambda \nu_d r^d \log n}{2} (M_Q(t)-1) \bigg) \\
    &= \exp\bigg( \frac{\lambda \nu_d r^d \log n}{2}(M_P(t) + M_Q(t) - 2) \bigg).
\end{aligned}
\end{equation}

\noindent Now, we need to find the MGF's of $P$ and $Q$. We proceed by direct calculation. Observe that
\begin{equation*}
\begin{aligned}
    M_P(t) &= \E[e^{tP}] \\
    &= \E\bigg[\exp\bigg(t\log\bigg(\frac{f_\text{in}(D)}{f_\text{out}(D)}\bigg) U_\text{out} + t\log\bigg(\frac{1-f_\text{in}(D)}{1-f_\text{out}(D)}\bigg)(1-U_\text{out}) \bigg)\bigg] \\
    &= \E\bigg[ \bigg(\frac{f_\text{in}(D)}{f_\text{out}(D)}\bigg)^{tU_\text{out}} \bigg(\frac{1-f_\text{in}(D)}{1-f_\text{out}(D)}\bigg)^{t(1-U_\text{out})} \bigg] \\
    &= \E\bigg[ \E\bigg[ (\frac{f_\text{in}(D)}{f_\text{out}(D)})^{tU_\text{out}} (\frac{1-f_\text{in}(D)}{1-f_\text{out}(D)})^{t(1-U_\text{out})} \: \bigg| \: D \bigg] \bigg] \text{ by the law of total expectation} \\
    &= \int_0^r \E\bigg[ \bigg(\frac{f_\text{in}(D)}{f_\text{out}(D)}\bigg)^{tU_\text{out}} \bigg(\frac{1-f_\text{in}(D)}{1-f_\text{out}(D)}\bigg)^{t(1-U_\text{out})} \: \bigg| \: D = x \bigg] \frac{dx^{d-1}}{r^d} dx.
\end{aligned}
\end{equation*}
To calculate the conditional expectation, we use the fact that $(U_\text{out} \mid D = x) \sim \text{Bern}(f_\text{out}(x))$ to obtain
\begin{equation*}
\begin{aligned}
    \E\bigg[ \bigg(\frac{f_\text{in}(D)}{f_\text{out}(D)}\bigg)^{tU_\text{out}} \bigg(\frac{1-f_\text{in}(D)}{1-f_\text{out}(D)}\bigg)^{t(1-U_\text{out})} \: \bigg| \: D = x \bigg] &= \bigg(\frac{f_\text{in}(x)}{f_\text{out}(x)}\bigg)^t f_\text{out}(x) + \bigg(\frac{1-f_\text{in}(x)}{1-f_\text{out}(x)}\bigg)^t (1 - f_\text{out(x)}) \\
    &= f_\text{in}(x)^t f_\text{out}(x)^{1-t} + (1-f_\text{in}(x))^t(1-f_\text{out}(x))^{1-t}.
\end{aligned}
\end{equation*}
Therefore, the MGF of $P$ is
\begin{equation}
\label{eq:MGF_P}
    M_P(t) = \int_0^r \left( f_\text{in}(x)^t f_\text{out}(x)^{1-t} + (1-f_\text{in}(x))^t(1-f_\text{out}(x))^{1-t} \right) \frac{dx^{d-1}}{r^d} dx.
\end{equation}
By a similar computation, the MGF of $Q$ is
\begin{equation}
\label{eq:MGF_Q}
    M_Q(t) = \int_0^r \left( f_\text{out}(x)^t f_\text{in}(x)^{1-t} + (1-f_\text{out}(x))^t(1-f_\text{in}(x))^{1-t} \right) \frac{dx^{d-1}}{r^d} dx.
\end{equation}
We note that when $t = 1/2$,
\begin{equation}
\label{eq:MGF_P/2,Q/2}
    M_P\left(\frac{1}{2}\right) = M_Q\left(\frac{1}{2}\right) = \int_0^r \bigg( \sqrt{f_\text{in}(x) f_\text{out}(x)} + \sqrt{(1-f_\text{in}(x))(1-f_\text{out}(x))} \bigg) \frac{dx^{d-1}}{r^d} dx.
\end{equation}
Then, we substitute \eqref{eq:MGF_P/2,Q/2} into \eqref{eq:MGF_Z} to obtain
\begin{equation}
\label{eq:MGF_Z/2}
\begin{aligned}
    \E[e^{Z/2}] &= \exp\bigg( \lambda \nu_d r^d \log n \int_0^r \left( \sqrt{f_\text{in}(x) f_\text{out}(x)} + \sqrt{(1-f_\text{in}(x))(1-f_\text{out}(x))} -1 \right) \frac{dx^{d-1}}{r^d} dx \bigg) \\
    &= \exp(-I(f_\text{in}, f_\text{out}) \log n)\\
    &= n^{-I(f_\text{in}, f_\text{out})}.
\end{aligned}
\end{equation}
Finally, we can substitute (\ref{eq:MGF_Z/2}) into the Chernoff bound (\ref{eq:Chernoff_Z}), which yields
\begin{equation*}
    \P(Z \geq -\rho \eta \log n) \leq n^{\rho \eta / 2} \E[e^{Z/2}] = n^{-(I(f_\text{in}, f_\text{out}) - \rho \eta /2)}
\end{equation*}
as desired. Note that we obtain \eqref{eq:tau-_incorrect_prob} since $(\tau(v, \sigma^*) \mid \sigma^*(v) = -1) \sim Z$ and \eqref{eq:tau+_incorrect_prob} since $(\tau(v, \sigma^*) \mid \sigma^*(v) = +1) \sim -Z$.
\end{proof}

Using Lemma \ref{lem:refine_Chernoff}, we can establish that $\tilde{\sigma}$ achieves exact recovery with probability $1 - o(1)$. We adapt the proof of Theorem 2.2 in \cite{Gaudio+2024} to our setting.

\begin{proof}[Proof of Theorem \ref{thm:achievability}]
Fix $\eta > 0$ and let the event $\mathcal{E}_1$ be the event that $\hat{\sigma}$ makes at most $\eta \log n$ mistakes in the neighborhood of every vertex $v \in V$. That is,
\begin{equation*}
    \mathcal{E}_1 = \bigcap_{v \in V}\{u \in \calN(v):  |\hat{\sigma}(u) \neq \sigma^*(u) \sigma^*(u_0)| < \eta \log n\}.
\end{equation*}
By Theorem \ref{thm:propagation_error_total}, $\P(\mathcal{E}_1) = 1 - o(1)$. 

We will provide an upper bound for the probability that taking the sign of $\tau(v, \sigma^*(u_0) \sigma)$ results in an incorrect community label for a vertex $v$, which is uniform over all labelings $\sigma$ that are ``close'' to the true labels $\sigma^*$. Let $W(v; \eta)$ be the set of labelings $\sigma$ that differ from the true labels $\sigma^*$ by at most $\eta \log n$ vertices in $\calN(v)$ (up to a global sign flip). That is,
\begin{equation*}
    W(v; \eta) = \{\sigma : |\{u \in \calN(v): \sigma(u) \neq \sigma^*(u_0) \sigma^*(u)\}| \leq \eta \log n\}.
\end{equation*}
Note that $\hat{\sigma} \in W(v; \eta)$ with probability $1 - o(1)$ due to Theorem \ref{thm:propagation_error_total}. Now, define the event
\begin{equation*}
\begin{aligned}
    \calE_v &= \bigg( \{\sigma^*(v)=+1\} \bigcap \bigg( \bigcup_{\sigma \in W(v; \eta)} \{ \tau(v, \sigma^*(u_0) \sigma) \leq 0 \} \bigg) \bigg) \\
    &\qquad \bigcup \bigg( \{\sigma^*(v)=-1\} \bigcap \bigg( \bigcup_{\sigma \in W(v; \eta)} \{ \tau(v, \sigma^*(u_0) \sigma) \geq 0 \} \bigg) \bigg).
\end{aligned}
\end{equation*}
To interpret $\mathcal{E}_v$, suppose that we do not know $\sigma(v)$ but we estimate it given the remaining data using the maximum likelihood estimator. Then, $\mathcal{E}_v$ is the event that there exists an ``almost-correct'' labeling $\sigma$ (up to a global sign flip) for which the likelihood ratio $\tau(v, \sigma^*(u_0) \sigma)$ produces the incorrect label for $v$.

Now, let $\calE_2$ be the event that all vertices are labeled correctly relative to $u_0$, meaning that
\begin{equation*}
    \calE_2 = \bigcap_{v \in V} \{ \tilde{\sigma}(v) = \sigma^*(u_0) \sigma^*(v) \}.
\end{equation*}
Following similar steps to the proof of \cite[Theorem 2.2]{Gaudio+2024}, we have that
\begin{equation}
\label{eq:prob_E2_complement}
    \mathbb{P}(\calE_2^c) \leq \sum_{i=1}^{cn} \mathbb{P}(\calE_{v_i}) + o(1),
\end{equation}
where $c$ is large enough so that $\mathbb{P}(|V| \leq cn) = 1-o(1)$. Here $v_1, v_2, \ldots$ is an arbitrary enumeration of the vertices, beginning with the first $N \sim \text{Poisson}(\lambda n)$ vertices and adding ``dummy'' vertices as needed. These dummy vertices are not included when estimating the labels of the true vertices. Thus, it remains to show $\mathbb{P}(\calE_{v_i}) = o(1/n)$.

We first show that if $\sigma \in W(v; \eta)$, then $| \tau(v, \sigma^*(u_0)\sigma) - \tau(v, \sigma^*) | \leq \rho \eta \log n$ for some $\rho, \eta > 0$. Let us define the sets
\begin{equation*}
\begin{aligned}
    &R_-(v) := \{u \in V: \sigma^*(u_0)\sigma(u) = +1, \sigma^*(u) = -1, u \visible v \}, \\
    &R_+(v) := \{u \in V : \sigma^*(u_0) \sigma(u) = -1, \sigma^*(u) = +1, u \visible v\}, \text{ and} \\
    &R(v) := R_+(v) \cup R_-(v) = \{u \in V: \sigma^*(u_0)\sigma(u) \neq \sigma^*(u), u \visible v \}.
\end{aligned}
\end{equation*}
By the definition of $\tau(v, \sigma^*(u_0)\sigma)$ and $\tau(v, \sigma^*)$, we obtain that
\begin{equation*}
\begin{aligned}
    &\tau(v, \sigma^*(u_0)\sigma) - \tau(v, \sigma^*) \\
    &\qquad= 2 \sum_{u \in R_-(v)}  \left( \log\bigg(\frac{\overline{f}_\text{in}(\|u-v\|)}{\overline{f}_\text{out}(\|u-v\|)}\bigg) \1(u \sim v) + \log\bigg(\frac{1 - \overline{f}_\text{in}(\|u-v\|)}{1 - \overline{f}_\text{out}(\|u-v\|)}\bigg)\1(u \nsim v) \right) \\
    &\qquad \qquad- 2 \sum_{u \in R_+(v)} \left( \log\bigg(\frac{\overline{f}_\text{in}(\|u-v\|)}{\overline{f}_\text{out}(\|u-v\|)}\bigg) \1(u \sim v) + \log\bigg(\frac{1 - \overline{f}_\text{in}(\|u-v\|)}{1 - \overline{f}_\text{out}(\|u-v\|)}\bigg)\1(u \nsim v) \right).
\end{aligned}
\end{equation*}
Taking the absolute value and using the triangle inequality yields
\begin{equation*}
\begin{aligned}
    &|\tau(v, \sigma^*(u_0)\sigma) - \tau(v, \sigma^*)| \\
    &\qquad \leq 2 \sum_{u \in R_-(v)}  \bigg|  \log\bigg(\frac{\overline{f}_\text{in}(\|u-v\|)}{\overline{f}_\text{out}(\|u - v\|)}\bigg) \1(u \sim v) + \log\bigg(\frac{1 - \overline{f}_\text{in}(\|u-v\|)}{1 - \overline{f}_\text{out}(\|u-v\|)}\bigg)\1(u \nsim v) \bigg| \\
    &\qquad \qquad+ 2 \sum_{u \in R_+(v)} \bigg| \log\bigg(\frac{\overline{f}_\text{in}(\|u-v\|)}{\overline{f}_\text{out}(\|u - v\|)}\bigg) \1(u \sim v) + \log\bigg(\frac{1 - \overline{f}_\text{in}(\|u-v\|)}{1 - \overline{f}_\text{out}(\|u-v\|)}\bigg)\1(u \nsim v) \bigg| \\
    &\qquad= 2 \sum_{u \in R(v)} \bigg| \log\bigg(\frac{\overline{f}_\text{in}(\|u-v\|)}{\overline{f}_\text{out}(\|u - v\|)}\bigg) \1(u \sim v) + \log\bigg(\frac{1 - \overline{f}_\text{in}(\|u-v\|)}{1 - \overline{f}_\text{out}(\|u-v\|)}\bigg)\1(u \nsim v) \bigg| \\
    &\qquad\leq 2 \sum_{u \in R(v)} \bigg| \log\bigg(\frac{\overline{f}_\text{in}(\|u-v\|)}{\overline{f}_\text{out}(\|u - v\|)}\bigg) \1(u \sim v) \bigg| + \bigg| \log\bigg(\frac{1 - \overline{f}_\text{in}(\|u-v\|)}{1 - \overline{f}_\text{out}(\|u-v\|)}\bigg)\1(u \nsim v) \bigg| \\
    &\qquad \leq 2 \sum_{u \in R(v)} \bigg| \log\bigg(\frac{\overline{f}_\text{in}(\|u-v\|)}{\overline{f}_\text{out}(\|u - v\|)}\bigg) \bigg| + \bigg| \log\bigg(\frac{1 - \overline{f}_\text{in}(\|u-v\|)}{1 - \overline{f}_\text{out}(\|u-v\|)}\bigg) \bigg|.
\end{aligned} 
\end{equation*}
Recall that we have $\xi < f_\text{in}(t), f_\text{out}(t) < 1-\xi$ for all $t \leq r$ from Assumption \ref{eq:function_assumption_2}, which implies that 
\begin{equation*}
    \bigg| \log\bigg(\frac{\overline{f}_\text{in}(\|u-v\|)}{\overline{f}_\text{out}(\|u - v\|)}\bigg) \bigg| < \log(1/\xi) \quad \text{and} \quad \bigg| \log\bigg(\frac{1 - \overline{f}_\text{in}(\|u-v\|)}{1 - \overline{f}_\text{out}(\|u-v\|)}\bigg) \bigg| < \log(1/\xi).
\end{equation*}
Therefore,
\begin{equation*}
    |\tau(v, \sigma^*(u_0)\sigma) - \tau(v, \sigma^*)| \leq 2 \sum_{u \in R(v)} 2 \log(1/\xi) \leq 4 \log(1/\xi) \eta \log n,
\end{equation*}
where last inequality holds because $\sigma \in W(v; \eta)$, so there are at most $\eta \log n$ vertices $u$ such that $\sigma^*(u_0)\sigma(u) \neq \sigma^*(u)$. Finally, we can let $\rho := 4 \log(1/\xi)$ to obtain 
\begin{equation}
\label{eq:tau_difference}
    |\tau(v, \sigma^*(u_0)\sigma) - \tau(v, \sigma^*)| \leq \rho \eta \log n.
\end{equation}

Now, we show that $\P(\mathcal{E}_v) = o(1/n)$. By the law of total probability,
\begin{equation*}
    \P(\mathcal{E}_v) = \frac{1}{2}\P(\mathcal{E}_v \mid \sigma^*(v) = +1) + \frac{1}{2}\P(\mathcal{E}_v \mid \sigma^*(v) = -1).
\end{equation*}
We consider each conditional probability separately. Observe that
\begin{equation*}
\begin{aligned}
    \P(\mathcal{E}_v \mid \sigma^*(v) = -1) &= \P\bigg(\bigcup_{\sigma \in W(v; \eta)} \{\tau(v, \sigma^*(u_0) \sigma) \geq 0\} \mid \sigma^*(v) = -1 \bigg) \\
    &= \P\big(\max_{\sigma \in W(v; \eta)} \tau(v, \sigma^*(u_0) \sigma) \geq 0 \mid \sigma^*(v) = -1 \big) \\
    &\leq \P\big(\tau(v, \sigma^*) \geq - \rho \eta \log n \mid \sigma^*(v) = -1 \big).
\end{aligned}
\end{equation*}
where the last inequality holds because $\max_{\sigma \in W(v; \eta)} \tau(v, \sigma^*(u_0) \sigma) \geq 0$ implies $\tau(v, \sigma^*) \geq -\rho \eta \log n$ due to $|\tau(v, \sigma^*(u_0) \sigma) - \tau(v, \sigma^*)| \leq \rho \eta \log n$ for all $\sigma \in W(v;\eta)$. Then, note that $(\tau(v, \sigma^*) \mid \sigma^*(v) = -1) \sim Z$ as defined in \eqref{eq:def_Z}. Therefore,
\begin{equation*}
\begin{aligned}
    &\P(\mathcal{E}_v \mid \sigma^*(v) = -1) \leq \P(\tau(v, \sigma^*) \geq -\rho \eta \log n \mid \sigma^*(v)= -1) = \P(Z \geq -\rho \eta \log n) \leq n^{-(I(f_\text{in}, f_\text{out}) - \rho \eta /2)},
\end{aligned}
\end{equation*}
where the last inequality follows from Lemma \ref{lem:refine_Chernoff}. Similarly, we can show that
\begin{equation*}
\begin{aligned}
    &\P(\mathcal{E}_v \mid \sigma^*(v) = +1) \leq \P(\tau(v, \sigma^*) \leq \rho \eta \log n \mid \sigma^*(v) = +1) = \P(-Z \leq \rho \eta \log n) \leq n^{-(I(f_\text{in}, f_\text{out}) - \rho \eta /2)}.
\end{aligned}
\end{equation*}
Therefore, we obtain
\begin{equation}
\label{eq:prob_Ev_final}
    \P(\calE_v) \leq n^{-(I(f_\text{in}, f_\text{out}) - \rho \eta /2)}.
\end{equation}
Since $I(f_\text{in}, f_\text{out}) > 1$ and \eqref{eq:prob_Ev_final} holds for any fixed $\eta > 0$, we can let $\eta = (I(f_\text{in}, f_\text{out}) - 1) / 2 > 0$, which establishes that $\P(\calE_v) = o(1/n)$. Then, we obtain that $\P(\calE_2) = 1 - o(1)$ from \eqref{eq:prob_E2_complement}, which shows that $\tilde{\sigma}$ achieves exact recovery.
\end{proof}

\section{Proof of Impossibility}
\label{sec:Impossibility}
In this section, we show that exact recovery is impossible under the conditions specified in Theorem \ref{thm:impossibility}. We first consider the case where $I(f_\text{in}, f_\text{out}) < 1$. To show that exact recovery is impossible, we will prove that the maximum likelihood estimator (MLE) fails with a constant, nonzero probability. This implies that any estimator will fail with probability bounded away from zero since the MLE is optimal\footnote{The MLE coincides with the MAP estimator, which is optimal, so the MLE is optimal as well.}.
Now, to prove that the MLE fails, it is sufficient to show that there exists a vertex $v$ such that 
\begin{equation*}
    \sign(\tau(v, \sigma^*)) \neq \sigma^*(v)
\end{equation*}
where $\tau(v, \sigma)$ is defined in Equation (\ref{eq:def_tau}). That is, changing the community label of $v$ either increases the likelihood of the observed graph or keeps the likelihood the same. We call such a vertex ``Flip-Bad'', a term which was introduced in \cite{Abbe+2020}.

We will start with the proof for the condition $I(f_\text{in}, f_\text{out}) < 1$, splitting the proof into two cases: $\lambda \nu_d r^d < 1$ and $\lambda \nu_d r^d \geq 1$.

Case 1: $\lambda \nu_d r^d < 1$. We define the vertex visibility graph of $G$ as the graph formed from taking the vertices of $G$ and adding an edge between any pair of vertices with distance less than $r(\log n)^{1/d}$. We will show that the vertex visibility graph is disconnected with high probability; if the vertex visibility graph is disconnected, then changing the community labels of all vertices in one connected component keeps the likelihood of the observed graph the same. Hence, there would exist a Flip-Bad vertex and the MLE would therefore fail. 

We note that the vertex visibility graph is a \textit{random geometric graph}. A random geometric graph, which we denote $RGG(d, n, x)$, is defined as follows: First, we distribute vertices within the $d$-dimensional unit cube by a Poisson point process of intensity $n$. Then, for any pair of vertices, we form an edge if and only if the distance between the vertices is less than $x$. We remark that there are multiple commonly-used definitions for a random geometric graph; the most common definition distributes a fixed number of vertices $n$ uniformly at random rather than using a Poisson point process of intensity $n$. Therefore, we can establish the disconnectivity of the vertex visibility graph using results from random geometric graphs literature. In particular, we rely on the following result from Penrose \cite{Penrose1997}.

\begin{lemma}[Result (1) from \cite{Penrose1997}]
\label{lem:RGG_connectivity_lemma}
    Consider a set of vertices distributed in the $d$-dimensional unit cube by a Poisson process of intensity $n$, and let $M_n := \inf\{x: RGG(d, n, x) \text{ is connected}\}$ be the minimum visibility radius such that the random geometric graph is connected. If $d \geq 2$, then for any fixed $\alpha \in \R$,
    \begin{equation*}
        \lim_{n \to \infty} \P\bigg(M_n \leq \Big(\frac{\alpha + \log n}{n \nu_d}\Big)^{1/d}\bigg) = \exp(e^{-\alpha}).
    \end{equation*}
\end{lemma}

Recall that $G$ is formed by distributing points in a volume $n$ cube by a Poisson point process with intensity $\lambda$. To apply Lemma \ref{lem:RGG_connectivity_lemma}, we must rescale $G$ by a factor of $1/n^{1/d}$. Then, the vertex visibility graph becomes the random geometric graph $\text{RGG}(d, \lambda n, r(\log(n)/n)^{1/d})$. Now, define $M_n := \inf\{x : \text{RGG}(d, \lambda n, r(\log(n)/n)^{1/d}) \text{ is connected}\}$ and note that the vertex visibility graph is connected if and only if $M_n \leq r(\log(n)/n)^{1/d}$. Thus, we can show that the vertex visibility graph is disconnected with high probability by showing
\begin{equation}
\label{eq:Mn_disconnectivity_threshold}
    \lim_{n \to \infty} \P\bigg(M_n \leq r\Big(\frac{\log n}{n}\Big)^{1/d} \bigg) = 0
\end{equation}

\begin{prop}
\label{prop:impossibility_disconnected_graph}
    Let $G \sim \text{GSBM}(\lambda, r, n, f_\text{in}, f_\text{out}, d)$. Suppose that $d \geq 2$ and $\lambda \nu_d r^d < 1$. Then, the vertex visibility graph is disconnected with high probability. Consequently, the MLE fails with a constant, nonzero probability.
\end{prop}
\begin{proof}
To show \eqref{eq:Mn_disconnectivity_threshold}, we will show that for any $\alpha \in \R$,
\begin{equation*}
    \lim_{n \to \infty} \P\bigg(M_n \leq r\Big(\frac{\log n}{n}\Big)^{1/d}\bigg) \leq \exp(-e^{-\alpha}).
\end{equation*}
Fix $\alpha \in \R$. Observe that since $\lambda \nu_d r^d < 1$, we have $\log n \leq (\alpha + \log \lambda + \log n)/(\lambda \nu_d r^d)$ for sufficiently large $n$. Therefore, we have
\begin{equation*}
    r \Big(\frac{\log n}{n}\Big)^{1/d} \leq \Big(\frac{\alpha + \log (\lambda n)}{n \lambda \nu_d}\Big)^{1/d}
\end{equation*}
which implies that
\begin{equation*}
    \P\bigg(M_n \leq r \Big(\frac{\log n}{n}\Big)^{1/d}\bigg) \leq \P\bigg(M_n \leq \Big(\frac{\alpha + \log (\lambda n)}{n \lambda \nu_d}\Big)^{1/d}\bigg).
\end{equation*}
Then, taking the limit as $n \to \infty$ and applying Lemma \ref{lem:RGG_connectivity_lemma} yields
\begin{equation*}
    \lim_{n \to \infty} \P\bigg(M_n \leq r\Big(\frac{\log n}{n}\Big)^{1/d}\bigg) \leq \exp(-e^{-\alpha}).
\end{equation*}
Since this holds for all $\alpha \in \R$, we obtain
\begin{equation*}
    \lim_{n \to \infty} \P\bigg(M_n \leq r \Big( \frac{\log n}{n}\Big)^{1/d}\bigg) = 0
\end{equation*}
Then, since the vertex visibility graph is connected if and only if $M_n \leq r(\log(n)/n)^{1/d}$, we see that the vertex visibility graph is disconnected with high probability.
\end{proof}

Case 2: $\lambda \nu_d r^d \geq 1$. The following lemma gives a sufficient condition for the MLE to fail.

\begin{lemma}[Proposition 8.1 in \cite{Abbe+2020}]
\label{lem:impossibility_lemma}
    Suppose that the graph $G$ satisfies the following two conditions.
    \begin{gather}
    \label{eq:impossibility_condition_1} \lim_{n\to\infty} n\E^{0}[\1_{0 \text{ is Flip-Bad in } G \cup \{0\}}] = \infty \\
    \label{eq:impossibility_condition_2} \limsup_{n\to\infty} \frac{\int_{y \in \calS_{d,n}} \E^{0,y}[\1_{0 \text{ is Flip-Bad in } G \cup \{0, y\}} \1_{y \text{ is Flip-Bad in } G \cup \{0, y\}} ] \: m_{n,d}(dy)}{n\E^{0}[\1_{0 \text{ is Flip-Bad in } G \cup \{0\}}]^2} \leq 1
    \end{gather}
    where $m_{n,d}$ is the Haar measure. Then, exact recovery is impossible. 
    
    $\E^0$ represents the expectation with respect to the graph $G \cup \{0\}$ where the community label of $0$ is sampled uniformly from $\{+1, -1\}$ and the edges from $0$ to all other vertices is sampled according to the model. Similarly, $\E^{0,y}$ is the expectation with respect to the graph $G \cup \{0, y\}$ where the community label of $0$ and $y$ are sampled uniformly from $\{+1, -1\}$ and the edges from $0$ and $y$ to all other vertices is sampled according to the model. 
\end{lemma}

We first show that (\ref{eq:impossibility_condition_1}) holds. We will adapt the proof of Lemma B.4 in \cite{Gaudio+2025}, which uses techniques from large deviations to lower-bound $\E^0[\1_{0 \text{ is Flip-Bad in } G \cup \{0\}}]$.

\begin{definition}[Rate Function]
    Let $X$ be any random variable. We define the rate function of $X$, denoted $\Lambda_X^*$, as
    \begin{equation*}
        \Lambda_X^*(\alpha) := \sup_{t \in \R} (t\alpha - \Lambda_X(t))
    \end{equation*}
    where $\Lambda_X(t) := \log \E[\exp(tX)]$ is the cumulant generating function of $X$.
\end{definition}

\begin{lemma}[Cramer's Theorem]
\label{lem:Cramer}
Let $\{X_i\}$ be a sequence of i.i.d. random variables with rate function $\Lambda_X^*$. Then, for all $\alpha > \E[X_1]$,
    \begin{equation*}
        \lim_{m \to \infty} \frac{1}{m} \log \P(\sum_{i=1}^m X_i \geq m\alpha) = -\Lambda^*_X(\alpha).
    \end{equation*}
\end{lemma}

\begin{lemma}
\label{lem:impossibility_lemma_1}
    Let $G \sim \text{GSBM}(\lambda, r, n, f_\text{in}, f_\text{out}, d)$, and suppose that $\lambda \nu_d r^d \geq 1$ and $I(f_\text{in}, f_\text{out}) < 1$. Then, there exists a constant $\beta > 0$ such that
    \begin{equation*}
        \E^0[\1_{0 \text{ is Flip-Bad in } G \cup \{0\}}] > n^{-1 + \beta},
    \end{equation*}
    which implies that \eqref{eq:impossibility_condition_1} is satisfied. 
\end{lemma}
\begin{proof}
    Recall that the log-likelihood ratio for the label of $0$ is given by $\tau(0, \sigma^*)$, where $\tau(0, \sigma^*) > 0$ indicates a greater likelihood for $\sigma^*(0) = +1$ and $\tau(0, \sigma^*) < 0$ indicates a greater likelihood for $\sigma^*(0) = -1$. Thus, the probability that $0$ is Flip-Bad is
    \begin{equation}
    \label{eq:flip_bad_probability}
        \E[\1_{0 \text{ is Flip-Bad in } G \cup \{0\} }] = \frac{1}{2} \P(\tau(0, \sigma^*) \leq 0 \mid \sigma^*(0) = +1) + \frac{1}{2}\P(\tau(0, \sigma^*) \geq 0 \mid \sigma^*(0) = -1).
    \end{equation}
    We first show that $\P(\tau(0, \sigma^*) \geq 0 \mid \sigma^*(0) = -1) > n^{-1 + \beta}$. Let $N(0)$ be the number of vertices visible to $0$. By the law of total probability,
    \begin{equation}
    \label{eq:MLE_error_probability}
        \P(\tau(0, \sigma^*) \geq 0 \mid \sigma^*(0) = -1) = \sum_{m=0}^\infty \P(\tau(0, \sigma^*) \geq 0 \mid \sigma^*(0) = -1, N(0) = m) \P(N(0) = m).
    \end{equation}
    We focus on analyzing $\P(\tau(0, \sigma^*) \geq 0 \mid \sigma^*(0) = -1, N(0) = m)$. Let $\{W_i\}$ be a sequence of i.i.d. random variables from the distribution $W$, which we construct as follows: Consider $P$ and $Q$ as defined in (\ref{eq:def_P}) and (\ref{eq:def_Q}), respectively. Then, let $W$ be the mixture distribution $W = TP + (1-T)Q$ where $T$ is an independent $\text{Bern}(1/2)$ random variable. That is, $W = P$ with probability 1/2 and $W = Q$ with probability 1/2. 
    
    From the definitions of $\tau(0, \sigma^*)$ and $W$, we have
    \begin{equation*}
        \left( \tau(0, \sigma^*) \mid \sigma^*(0)=-1, N(0) = m \right) \sim \sum_{i=1}^m W_i
    \end{equation*}
    which implies that
    \begin{equation} \label{eq:conditional_MLE_error_probability}
        \P(\tau(0, \sigma^*) \geq 0 \mid \sigma^*(0)=-1, N(0) = m) = \P\Big(\sum_{i=1}^m W_i \geq 0\Big).
    \end{equation}
    We will bound the probability in \eqref{eq:conditional_MLE_error_probability} using Cramer's Theorem (Lemma \ref{lem:Cramer}). First, we show that $\E[W] < 0$. Computing $\E[W]$, we have that
    \begin{equation*}
    \begin{aligned}
        \E[W] &= \frac{1}{2}(\E[P] + \E[Q]) \\ 
        &= \frac{1}{2} \E\left[ \log\left(\frac{f_\text{in}(D)}{f_\text{out}(D)}\right)(f_\text{out}(D) - f_\text{in}(D)) + \log\left(\frac{1-f_\text{in}(D)}{1-f_\text{out}(D)}\right)(f_\text{in}(D) - f_\text{out}(D)) \right].
    \end{aligned}
    \end{equation*}
    where $D$ is defined as in \eqref{eq:def_D}. Then, we use the fact that for any constants $0 < a,b < 1$,
    \begin{equation*}
        \log\left(\frac{a}{b}\right)(b-a) + \log\left(\frac{1-a}{1-b}\right)(a-b) < 0,
    \end{equation*}
    which shows that $E[W] < 0$. Therefore, we can apply Cramer's Theorem (Lemma \ref{lem:Cramer}) with $\alpha = 0$, which yields
    \begin{equation*}
        \lim_{m \to \infty} \frac{1}{m} \log \P(\sum_{i=1}^m W_i \geq 0) = -\Lambda_W^*(0)
    \end{equation*}
    where $\Lambda_W^*$ is the rate function of $W$. This gives us a lower bound for $\P(\sum_{i=1}^m W_i \geq 0)$ because for any fixed $\epsilon > 0$, there exists $M$ such that
    \begin{equation*}
        \frac{1}{m} \log \P(\sum_{i=1}^m W_i \geq 0) > -\Lambda_W^*(0) - \epsilon \quad \text{for all }m \geq M
    \end{equation*}
    which implies that
    \begin{equation}
    \label{eq:Cramer_error_probability}
        \P(\sum_{i=1}^m W_i \geq 0) > \exp\Big(m(-\Lambda_W^*(0) - \epsilon)\Big) \quad \text{for all }m \geq M.
    \end{equation}

    Now, we return to simplifying Equation (\ref{eq:MLE_error_probability}). Let $\delta > 0$ be a constant, which we will determine later. Then, we can substitute (\ref{eq:conditional_MLE_error_probability}) into (\ref{eq:MLE_error_probability}) and use (\ref{eq:Cramer_error_probability}) to obtain
    \begin{align}
    \label{eq:MLE_error_probability_simplified} 
        \P(\tau(0, \sigma^*) \geq 0 \mid \sigma^*(0) = -1) &= \sum_{m=0}^\infty \P\bigg(\sum_{i=1}^m W_i \geq 0\bigg) \P(N(0) = m) \notag \\
        &\geq \sum_{m=\delta \log n}^\infty \P\bigg(\sum_{i=1}^m W_i \geq 0\bigg) \P(N(0) = m) \notag \\
        &\geq \sum_{m=\delta \log n}^\infty \exp\Big(m(-\Lambda_W^*(0) - \epsilon)\Big) \P(N(0) = m) \notag \\
        &= \sum_{m=0}^\infty \exp\Big(m(-\Lambda_W^*(0) - \epsilon)\Big) \P(N(0) = m) \notag \\
        &\qquad- \sum_{m=0}^{\delta \log n - 1} \exp\Big(m(-\Lambda_W^*(0) - \epsilon)\Big) \P(N(0) = m) \notag \\
        &\geq \bigg[ \sum_{m=0}^\infty \exp\Big(m(-\Lambda_W^*(0) - \epsilon)\Big) \P(N(0) = m) \bigg] - \P(N(0) < \delta \log n) \notag \\
        &= \E\Big[\exp\Big(N(0)(-\Lambda_W^*(0) - \epsilon) \Big)\Big] - \P(N(0) < \delta \log n).
    \end{align}
    The last inequality holds because the rate function $\Lambda_W^*(0)$ is non-negative, which implies that $0 < \exp\big(m(-\Lambda_W^*(0) - \epsilon)\big) < 1$. Now, recall that $N(0) \sim \text{Pois}(\lambda \nu_d r^d \log n)$, so its moment-generating function is 
    \begin{equation*}
        \E[\exp(tN(0))] = \exp\Bigl(\lambda \nu_d r^d (e^t -1) \log n\Bigr) = n^{\lambda \nu_d r^d(e^t - 1)}.
    \end{equation*}
    Letting $t = -\Lambda_W^*(0) - \epsilon$, we obtain that
    \begin{equation*}
        \E\Big[\exp\Big(N(0) (-\Lambda_W^*(0) - \epsilon) \Big)\Big] = n^{\lambda \nu_d r^d \big(\exp(-\Lambda_W^*(0) - \epsilon) - 1\big)}.
    \end{equation*}
    Therefore, (\ref{eq:MLE_error_probability_simplified}) becomes
    \begin{equation}
    \label{eq:MLE_error_probability_simplified_2}
        \P(\tau(0, \sigma^*) \geq 0 \mid \sigma^*(0) = -1) \geq n^{\lambda \nu_d r^d \big(\exp(-\Lambda_W^*(0) - \epsilon) - 1\big)} - \P(N(0) < \delta \log n).
    \end{equation}
    Next, we compute $\Lambda_W^*(0)$. By the definition of the rate function,
    \begin{equation*}
        \Lambda_W^*(0) = \sup_{t \in \R} -\Lambda_W(t) = - \inf_{t \in \R} \Lambda_W(t) = - \inf_{t \in \R} \log \E[e^{tW}]. 
    \end{equation*}
    Then, observe that
    \begin{equation*}
    \begin{aligned}
        \E[e^{tW}] &= \frac{1}{2}\E[e^{tP}] + \frac{1}{2}\E[e^{tQ}] \\
        & = \int_0^r \bigg( f_\text{in}(x)^t f_\text{out}(x)^{1-t} + (1-f_\text{in}(x))^t(1-f_\text{out}(x))^{1-t} \bigg) \frac{dx^{d-1}}{r^d} dx \\
        &\qquad+ \int_0^r \bigg( f_\text{out}(x)^t f_\text{in}(x)^{1-t} + (1-f_\text{out}(x))^t(1-f_\text{in}(x))^{1-t} \bigg) \frac{dx^{d-1}}{r^d} dx
    \end{aligned}
    \end{equation*}
    where $\E[e^{tP}]$ and $\E[e^{tQ}]$ were calculated in \eqref{eq:MGF_P} and \eqref{eq:MGF_Q}. Using the fact that 
    \begin{equation*}
        \argmin_{0 \leq t \leq 1} \: a^tb^{1-t} + b^ta^{1-t} + (1-a)^t(1-b)^{1-t}+(1-b)^t(1-a)^{1-t} = \frac{1}{2}
    \end{equation*}
    for any constants $0 < a,b < 1$, we obtain that $\E[e^{tW}]$ is minimized at $t = 1/2$. Therefore,
    \begin{equation*}
    \begin{aligned}
        \Lambda_W^*(0) &= - \log \E[e^{W/2}] \\
        &= -\log \bigg( \int_0^r \left( \sqrt{f_\text{in}(x) f_\text{out}(x)} + \sqrt{(1-f_\text{in}(x))(1-f_\text{out}(x))} \right) \frac{dx^{d-1}}{r^d} dx \bigg). \\
    \end{aligned}
    \end{equation*}
    Now, we use the condition $I(f_\text{in}, f_\text{out}) < 1$. Observe that
    \begin{equation*}
    \begin{aligned}
        I(f_\text{in}, f_\text{out}) &= \lambda \nu_d r^d \int_0^r \left(1- \sqrt{f_\text{in}(x) f_\text{out}(x)} - \sqrt{(1-f_\text{in}(x))(1-f_\text{out}(x))}\right) \frac{dx^{d-1}}{r^d} dx  \\
        &= \lambda \nu_d r^d - \lambda \nu_d r^d \int_0^r \left( \sqrt{f_\text{in}(x) f_\text{out}(x)} + \sqrt{(1-f_\text{in}(x))(1-f_\text{out}(x))} \right) \frac{dx^{d-1}}{r^d} dx  \\
        &= \lambda \nu_d r^d - \lambda \nu_d r^d \exp(-\Lambda_W^*(0)) \\
        &= \lambda \nu_d r^d \Big(1 - \exp(-\Lambda_W^*(0))\Big) \\
        &< 1. 
    \end{aligned}
    \end{equation*}
    Thus, we obtain that $\lambda \nu_d r^d \big( \exp(-\Lambda_W^*(0)) - 1\big) > -1$. Next, let $\beta = \big(1 + \lambda \nu_d r^d\big(\exp(-\Lambda_W^*(0)) - 1\big)\big)/3$ and note that $\beta > 0$ from the above calculation. Since
    \begin{equation*}
        \lim_{\epsilon \to 0} \lambda \nu_d r^d \Bigl(\exp(-\Lambda_W^*(0) - \epsilon) - 1\Bigr) = \lambda \nu_d r^d \Bigl(\exp(-\Lambda_W^*(0)) - 1\Bigr),
    \end{equation*}
    we can choose a sufficiently small $\epsilon$ such that
    \begin{equation}
    \label{eq:impossibility_criterion_consequence}
        \lambda \nu_d r^d \Bigl(\exp(-\Lambda_W^*(0) - \epsilon) - 1\Bigr) > -1 + 2\beta.
    \end{equation}
    Substituting (\ref{eq:impossibility_criterion_consequence}) into (\ref{eq:MLE_error_probability_simplified_2}), we obtain
    \begin{equation}
    \label{eq:MLE_error_probability_simplified_3}
        \P(\tau(0, \sigma^*) \geq 0 \mid \sigma^*(0) = -1) > n^{-1 + 2\beta} - \P(N(0) < \delta \log n).
    \end{equation}
    Now, let $c < 2\beta$ be a constant. We will show that $\P(N(0) < \delta \log n) \leq n^{-1 + c}$ for a suitable choice of $\delta$. Since $N(0) \sim \text{Pois}(\lambda \nu_d r^d \log n)$, we can use the Chernoff bound (Lemma \ref{lem:poisson_Chernoff}) to obtain that for any $\delta > 0$,
    \begin{equation*}
    \begin{aligned}
        \P(N(0) < \delta \log n) &\leq \exp\bigg(\delta \log n - \lambda \nu_d r^d \log n - \log(\delta / (\lambda \nu_d r^d)) \delta \log n\bigg) \\
        &= \exp\bigg( \Big(\delta - \lambda \nu_d r^d - \delta \log(\delta/(\lambda \nu_d r^d))\Big) \log n \bigg) \\
        &= n^{\delta - \lambda \nu_d r^d - \delta \log(\delta/(\lambda \nu_d r^d)} \\
        &= n^{f(\delta)}
    \end{aligned}
    \end{equation*}
    where $f(\delta) := \delta - \lambda \nu_d r^d - \delta \log(\delta/(\lambda \nu_d r^d)$. Then, observe that $\lim_{\delta \downarrow 0} f(\delta) = -\lambda \nu_d r^d$ and that $f(\delta)$ is strictly increasing when $0 < \delta < \lambda \nu_d r^d$. Since $\lambda \nu_d r^d \geq 1$, for sufficiently small $\delta$ we have that $f(\delta) < -1 + c$. Therefore, $\P(N(0) < \delta \log n) \leq n^{-1 + c}$. Consequently, \eqref{eq:MLE_error_probability_simplified_3} becomes
    \begin{equation}
    \label{eq:MLE_error_probability_final_1}
        \P(\tau(0, \sigma^*) \geq 0 \mid \sigma^*(0) = -1) > n^{-1 + 2\beta} - n^{-1 + c} > n^{-1 + \beta}.
    \end{equation}
    A similar argument shows that 
    \begin{equation}
    \label{eq:MLE_error_probability_final_2}
        \P(\tau(0, \sigma^*) \leq 0 \mid \sigma^*(0) = +1) > n^{-1+\beta}.
    \end{equation}
    Finally, we substitute (\ref{eq:MLE_error_probability_final_1}) and (\ref{eq:MLE_error_probability_final_2}) into (\ref{eq:flip_bad_probability}), which shows that
    \begin{equation*}
        \E[\1_{0 \text{ is Flip-Bad in } G \cup \{0\} }] > n^{-1 + \beta}.
    \end{equation*}
    Therefore, (\ref{eq:impossibility_condition_1}) is satisfied. 
\end{proof}

Now, we show that (\ref{eq:impossibility_condition_2}) holds. The proof is essentially identical to the proof of Lemma 8.3 from \cite{Abbe+2020}, with some minor modifications.

\begin{lemma}
\label{lem:impossibility_lemma_2}
    Let $G \sim \text{GSBM}(\lambda, r, n, f_\text{in}, f_\text{out}, d)$, and suppose that $\lambda \nu_d r^d \geq 1$ and $I(f_\text{in}, f_\text{out}) < 1$. Then, (\ref{eq:impossibility_condition_2}) is satisfied.
\end{lemma}
\begin{proof}
    Define $B(0, a) := \{x \in \calS_{d, n} : \|x\|_2 \leq a\}$ as the Euclidean ball around the origin with radius $a$. Then, observe that 
    \begin{equation}
    \label{eq:flip_bad_integral}
    \begin{aligned}
        &\int_{y \in \calS_{d,n}} \E^{0,y}[\1_{0 \text{ is Flip-Bad in } G \cup \{0, y\}} \1_{y \text{ is Flip-Bad in } G \cup \{0, y\}} ] \: m_{n,d}(dy) \\
        &= \int_{y \in B(0, 2 r(\log n)^{1/d})} \E^{0,y}[\1_{0 \text{ is Flip-Bad in } G \cup \{0, y\}} \1_{y \text{ is Flip-Bad in } G \cup \{0, y\}} ] \: m_{n,d}(dy) \\
        &\qquad+ \int_{y \in \calS_{d,n} \setminus B(0, 2 r(\log n)^{1/d})} \E^{0,y}[\1_{0 \text{ is Flip-Bad in } G \cup \{0, y\}} \1_{y \text{ is Flip-Bad in } G \cup \{0, y\}} ] \: m_{n,d}(dy) \\
        &\leq \int_{y \in B(0, 2 r(\log n)^{1/d})} \E^{0,y}[\1_{0 \text{ is Flip-Bad in } G \cup \{0, y\}}] \: m_{n,d}(dy) \\
        &\qquad+ \int_{y \in \calS_{d,n} \setminus B(0, 2 r(\log n)^{1/d})} \E^{0,y}[\1_{0 \text{ is Flip-Bad in } G \cup \{0, y\}} \1_{y \text{ is Flip-Bad in } G \cup \{0, y\}} ] \: m_{n,d}(dy) \\
        &= \int_{y \in B(0, 2 r(\log n)^{1/d})} \E^{0,y}[\1_{0 \text{ is Flip-Bad in } G \cup \{0, y\}}] \: m_{n,d}(dy) \\
        &\qquad+ \int_{y \in \calS_{d,n} \setminus B(0, 2 r(\log n)^{1/d})} \E^0[\1_{0 \text{ is Flip-Bad in } G \cup \{0\}}] \E^y[\1_{y \text{ is Flip-Bad in } G \cup \{0\}}] \: m_{n,d}(dy) \\ 
        &= \int_{y \in B(0, 2 r(\log n)^{1/d})} \E^{0,y}[\1_{0 \text{ is Flip-Bad in } G \cup \{0, y\}}] \: m_{n,d}(dy) \\
        &\qquad+ \int_{y \in \calS_{d,n} \setminus B(0, 2 r(\log n)^{1/d})} \E^0[\1_{0 \text{ is Flip-Bad in } G \cup \{0\}}]^2 \: m_{n,d}(dy).
    \end{aligned}
    \end{equation}
    The last two equalities follow from the fact that if $y \in \calS_{d,n} \setminus B(0, 2 r(\log n)^{1/d})$, then $\calN(0)$ and $\calN(y)$ are disjoint. Then, since the events $\{0 \text{ is Flip-Bad in } G \cup \{0, y\}\}$ and $\{y \text{ is Flip-Bad in } G \cup \{0, y\}\}$ only depend on $\calN(0)$ and $\calN(y)$ respectively, the events are independent and have the same probability.

    Now, we consider each term separately. For the first term, let $Y$ be a uniform random variable on $B(0, 2r(\log n)^{1/d})$ and let $A$ be the event that there exists a point in $B(0, 2(\log n)^{1/d})$ other than the origin. Then, we can write the first term as
    \begin{align}
    \label{eq:flip_bad_integral_term_1}
        \int_{y \in B(0, 2 r(\log n)^{1/d})} \E^{0,y}[\1_{0 \text{ is Flip-Bad in } G \cup \{0, y\}}] \: m_{n,d}(dy) &= \P^{0, Y}(0 \text{ is Flip-Bad in } G \cup \{0, Y\}) \notag \\
        &= \P^0(0 \text{ is Flip-Bad in } G \cup \{0\} \mid A) \notag \\
        &\leq \frac{\P^0(0 \text{ is Flip-Bad in } G \cup \{0\})}{\P(A)} \notag \\
        &= \frac{\E^0[\1_{0 \text{ is Flip-Bad in } G \cup \{0\}}]}{1 - n^{-\lambda  (2r)^d \nu_d \log n}}.
    \end{align}
    For the second term, observe that $\E^0[\1_{0 \text{ is Flip-Bad in } G \cup \{0\}}]^2$ is a constant in $y$. Hence, the second term becomes
    \begin{align}
    \label{eq:flip_bad_integral_term_2}
        &\int_{y \in \calS_{d,n} \setminus B(0, 2 r(\log n)^{1/d})} \E^0[\1_{0 \text{ is Flip-Bad in } G \cup \{0\}}]^2 \: m_{n,d}(dy) \notag \\
        &= \text{Vol}\big(\calS_{d,n} \setminus B(0, 2 r(\log n)^{1/d})\big) \E^0[\1_{0 \text{ is Flip-Bad in } G \cup \{0\}}]^2 \notag \\
        &= \left(n - (2r)^d \nu_d \log n\right) \E^0[\1_{0 \text{ is Flip-Bad in } G \cup \{0\}}]^2.
    \end{align}
    Therefore, we substitute (\ref{eq:flip_bad_integral_term_1}) and (\ref{eq:flip_bad_integral_term_2}) into (\ref{eq:flip_bad_integral}) to obtain 
    \begin{equation*}
    \begin{aligned}
        &\int_{y \in \calS_{d,n}} \E^{0,y}[\1_{0 \text{ is Flip-Bad in } G \cup \{0, y\}} \1_{y \text{ is Flip-Bad in } G \cup \{0, y\}}] \: m_{n,d}(dy) \\
        &\leq \frac{\E^0[\1_{0 \text{ is Flip-Bad in } G \cup \{0\}}]}{1 - n^{-\lambda \nu_d (2r)^d}} + \left(n - (2r)^d \nu_d \log n\right) \E^0[\1_{0 \text{ is Flip-Bad in } G \cup \{0\}}]^2
    \end{aligned}
    \end{equation*}
    which shows that
    \begin{equation*}
    \begin{aligned}
        &\frac{\int_{y \in \calS_{d,n}} \E^{0,y}[\1_{0 \text{ is Flip-Bad in } G \cup \{0, y\}} \1_{y \text{ is Flip-Bad in } G \cup \{0, y\}} ] \: m_{n,d}(dy)}{n\E^{0}[\1_{0 \text{ is Flip-Bad in } G \cup \{0\}}]^2} \\
        &\leq \frac{1}{n \E^{0,y}[\1_{0 \text{ is Flip-Bad in } G \cup \{0, y\}}] (1 - n^{-\lambda \nu_d (2r)^d })} +  \frac{n - (2r)^d \nu_d \log n}{n} \\
        &=1 + \frac{1}{n \E^{0,y}[\1_{0 \text{ is Flip-Bad in } G \cup \{0, y\}}] (1 - n^{-\lambda \nu_d (2r)^d })} -\frac{(2r)^d\nu_d \log n}{n}.
    \end{aligned}
    \end{equation*}
    By Lemma \ref{lem:impossibility_lemma_1}, we know that $\E^0[\1_{0 \text{ is Flip-Bad in } G \cup \{0\}}] \geq n^{-1 + \beta}$ for some constant $\beta > 0$ when $\lambda \nu_d r^d \geq 1$ and $I(f_\text{in}, f_\text{out}) < 1$. Therefore, 
    \begin{equation*}
    \begin{aligned}
        &\frac{\int_{y \in \calS_{d,n}} \E^{0,y}[\1_{0 \text{ is Flip-Bad in } G \cup \{0, y\}} \1_{y \text{ is Flip-Bad in } G \cup \{0, y\}} ] \: m_{n,d}(dy)}{n\E^{0}[\1_{0 \text{ is Flip-Bad in } G \cup \{0\}}]^2} \\
        &\leq 1 + \frac{1}{n^\beta(1 - n^{-\lambda \nu_d (2r)^d })} - \frac{(2r)^d\nu_d \log n}{n}.
    \end{aligned}
    \end{equation*}
    Finally, we take the limit as $n \to \infty$, which yields
    \begin{equation*}
        \limsup_{n \to \infty} \frac{\int_{y \in \calS_{d,n}} \E^{0,y}[\1_{0 \text{ is Flip-Bad in } G \cup \{0, y\}} \1_{y \text{ is Flip-Bad in } G \cup \{0, y\}} ] \: m_{n,d}(dy)}{n\E^{0}[\1_{0 \text{ is Flip-Bad in } G \cup \{0\}}]^2} \leq 1.
    \end{equation*}
    Therefore, (\ref{eq:impossibility_condition_2}) is satisfied.
    
\end{proof}

From Lemmas \ref{lem:impossibility_lemma_1} and \ref{lem:impossibility_lemma_2}, we see that the conditions of Lemma \ref{lem:impossibility_lemma} holds. Therefore, we can conclude that any estimator fails to achieve exact recovery.

\begin{prop}
\label{prop:impossibility_connected_graph}
    Let $G \sim GSBM(\lambda, r, n, f_\text{in}, f_\text{out}, d)$ and suppose that $I(f_\text{in}, f_\text{out}) < 1$. Then, any estimator $\tilde{\sigma}$ fails to achieve exact recovery.
\end{prop}

We now show that if $d = 1$ and $\lambda r < 1$, then exact recovery is impossible. We first use a result from random geometric graphs literature which provides the connectivity threshold for random geometric graphs using the Euclidean metric rather than the toroidal metric \eqref{eq:def_toroidal_metric}.
\begin{lemma}[Theorem 2.2 in \cite{Muthukrishnan2005}]
\label{lem:1d_connectivity_threshold_Euclidean}
    Consider a set of vertices distributed on the unit interval by a Poisson process of intensity $n$. Then, $\text{RGG}(1, n, (c \log n)/n)$ is connected with high probability if $c > 1$ and disconnected with high probability if $c < 1$, where distance is measured using the Euclidean metric.
\end{lemma}

Now, we show that random geometric graphs using the toroidal metric have the same connectivity threshold.
\begin{lemma}
\label{lem:1d_connectivity_threshold_toroidal}
    Consider a set of vertices distributed on the unit interval by a Poisson process of intensity $n$. Then, $\text{RGG}(1, n, (c \log n)/n)$ is connected with high probability if $c > 1$ and disconnected with high probability if $c < 1$, where distance is measured using the toroidal metric \eqref{eq:def_toroidal_metric}.
\end{lemma}
\begin{proof}
For $c > 1$ connectivity immediately follows from Lemma \ref{lem:1d_connectivity_threshold_Euclidean}.

For the disconnectivity result, fix $c < c' < 1$, and consider two random geometric graphs $G_1, G_2 \sim \text{RGG}(1, n/2, (2c \log n)/n)$ under the Euclidean metric. By Lemma \ref{lem:1d_connectivity_threshold_Euclidean} and the observation that $(2c \log n)/n < (2c' \log (n/2))/n$, we have that $G_1$ and $G_2$ are disconnected with high probability. Now, we adjoin the line segments containing $G_1$ and $G_2$ to form a random geometric graph $H$ on an interval of length 2, with intensity $n/2$ and visibility radius $(2c \log n)/n$. Then $H$ is disconnected in at least two places with high probability.
    We then rescale $H$ by a factor of $1/2$ to produce a random geometric graph $H'$, and note that $H' \sim \text{RGG}(1, n, (c \log n)/n)$. Since $H$ and $H'$ have the same connectivity, we obtain that $H'$ is disconnected in at least two places with high probability. This implies that $H'$ is disconnected under the toroidal metric as well.
\end{proof}

Using Lemma \ref{lem:1d_connectivity_threshold_toroidal}, we can show that the vertex visibility graph of $G$ is disconnected with high probability when $\lambda r < 1$. Consequently, the MLE fails with a constant, nonzero probability, which implies that exact recovery is impossible.

\begin{prop}
\label{prop:impossibility_1d}
    Let $G \sim \text{GSBM}(\lambda, r, n, f_\text{in}, f_\text{out}, d)$. If $d = 1$ and $\lambda r < 1$, then the vertex visibility graph of $G$ is disconnected with probability $1 - o(1)$. Therefore, any estimator $\hat{\sigma}$ fails to achieve exact recovery.
\end{prop}
\begin{proof}
    We first rescale space by a factor of $1 / n$. Then, the vertex visibility graph becomes the random geometric graph $\text{RGG}(1, \lambda n, (r \log n)/n)$. By Lemma \ref{lem:1d_connectivity_threshold_toroidal}, the vertex visibility graph is disconnected with high probability if the visibility radius equals $(c \log n) / (\lambda n)$ where $c < 1$. Hence, the vertex visibility graph is disconnected with high probability when $r < 1 / \lambda$, i.e. when $\lambda r < 1$. 

    Now, we show that the MLE fails with a constant, nonzero probability. Observe that if the vertex visibility graph is disconnected, then then the maximum likelihood estimator labels the vertices in each connected component independently from vertices in other components. Therefore, the MLE fails with nonzero probability because if we flip the labels of all vertices in a particular component, the likelihood remains the same. Formally, suppose that we have $k \geq 2$ components, which we call $C_1, C_2, \ldots, C_k$, and let $t \in \{\pm 1\}^k$. Then, any estimator $\sigma_t: V \to \{-1, +1\}$ defined
    \begin{equation*}
        \sigma_t(v) = \sigma^*(v) \sum_{i=1}^k t_i \sum_{j \in C_i} \1{\{v \in B_j\}}
    \end{equation*}
    has the same likelihood as the true labeling $\sigma^*$. Therefore, the error probability of the MLE is at least $1 - 1/2^{k-1}$, which implies that exact recovery is impossible.
    \end{proof}

Combining Propositions \ref{prop:impossibility_disconnected_graph}, \ref{prop:impossibility_connected_graph}, and \ref{prop:impossibility_1d} together yields Theorem \ref{thm:impossibility}.

\section{Conclusion and Future Work}
\label{sec:conclusion}
In this paper, we identified the information-theoretic threshold for exact recovery in the two-community symmetric GSBM, where the edge probabilities between two vertices depend on their distance as well as their communities. In addition, we provide an efficient two-phase algorithm which achieves exact recovery above this threshold. 

There are several directions for future work. One direction would be considering more general models, such as the GSBM with multiple communities and arbitrary edge weight distributions, and deriving the information-theoretic threshold for exact recovery in these models as well. Another direction for future research is relaxing Assumption \ref{assump:regularity_conditions}, the regularity conditions we impose on $f_\text{in}$ and $f_\text{out}$.

\bibliographystyle{plain}
\bibliography{Refs}

\end{document}